\documentclass[11pt,reqno]{amsart}

\usepackage[T1]{fontenc}
\usepackage{amsmath,amssymb,amsthm}
\usepackage{xcolor}
\usepackage[a4paper,margin=30mm]{geometry}
\usepackage[hidelinks]{hyperref}

\newtheorem{theorem}{Theorem}[section]
\newtheorem*{maintheorem}{Main Theorem}
\newtheorem*{question}{Question}
\newtheorem{proposition}[theorem]{Proposition}
\newtheorem{lemma}[theorem]{Lemma}
\newtheorem{corollary}[theorem]{Corollary}
\theoremstyle{definition}
\newtheorem{definition}[theorem]{Definition}
\theoremstyle{remark}
\newtheorem{remark}[theorem]{Remark}

\numberwithin{equation}{section}
\newcommand{\N}{\mathbb{N}}
\newcommand{\R}{\mathbb{R}}
\newcommand{\C}{\mathbb{C}}
\newcommand{\CP}{\mathbb{CP}}
\newcommand{\Sphere}{\mathbb{S}}
\newcommand{\dd}{\mathrm{d}}
\DeclareMathOperator{\SL}{SL}
\DeclareMathOperator{\SU}{SU}

\DeclareMathOperator{\End}{End}
\DeclareMathOperator{\ord}{ord}

\title[Non-algebraicity of Lawson surfaces]
{Non-algebraicity of the Lawson surfaces $\xi_{k,l}$}

\author[L. Heller]{Lynn Heller}
\address{Hetao Institute of Mathematics and Interdisciplinary Sciences, Shenzhen, China}
\email{lynn@himis-sz.cn}
\author[S. Heller]{Sebastian Heller}
\address{Department of Mathematics and Institute of Mathematical Sciences,
The Chinese University of Hong Kong}
\email{sheller@cuhk.edu.hk}

\date{}

\begin{document}

\setlength{\parindent}{0pt}
\setlength{\parskip}{\medskipamount}

\begin{abstract}
We prove that every Lawson surface $\xi_{k,l}\subset\Sphere^3$
of genus $kl>1$ is non-algebraic: any real polynomial vanishing
on $\xi_{k,l}$ vanishes identically on $\Sphere^3$.
This answers a question posed by Hsiang in 1967 and proves
Lawson's 1970 conjecture for $\xi_{2,2}$.
The proof uses a reflection principle for the monodromy of the
associated family of flat connections.
\end{abstract}

\maketitle

\section{Introduction}\label{sec:introduction}
Minimal surfaces in space forms are classical objects of differential
geometry. In the round $3$-sphere $\Sphere^3$, the simplest closed
examples are the totally geodesic sphere and the Clifford torus,
both of which are defined within $\Sphere^3$ by polynomial equations.
These are also the only closed embedded minimal surfaces of genus
zero or one: in genus zero, the Hopf differential vanishes, so the surface
is totally geodesic~\cite{Almgren}; in genus one,
Brendle~\cite{Brendle} resolved the Lawson conjecture by proving
that every embedded minimal torus is congruent to the Clifford torus
$\xi_{1,1}$.

In 1967, Hsiang~\cite[Problem~3, p.~263]{Hsiang} asked whether
spheres also admit closed embedded minimal hypersurfaces that are
not algebraic. The Lawson surfaces $\xi_{k,l}$, with integers
$k\geq l\geq1$, provide a concrete family to study this question.
Lawson~\cite{Lawson} constructed them in 1970 by analytic
continuation of solutions to Plateau problems with geodesic
quadrilateral boundaries. These were the first closed embedded
minimal surfaces in $\Sphere^3$ of genus greater than one.
The implicit nature of this construction leaves their algebraicity
unresolved; Lawson conjectured that $\xi_{2,2}$ is
non-algebraic~\cite[\S6, p.~350]{Lawson}.

In a 2021 MathOverflow discussion, Bryant~\cite{Bryant} described
non-algebraic immersed minimal tori in $\Sphere^3$ that are invariant
under a circle action. These examples are not embedded. In this
paper we prove that no Lawson surface of genus greater than one
is algebraic. This settles Lawson's conjecture for $\xi_{2,2}$
and answers Hsiang's question with embedded examples in $\Sphere^3$.

\begin{maintheorem}\hypertarget{thm:main}{}
For all integers $k\geq l\geq1$ with $kl>1$, the Lawson surface
$\xi_{k,l}$ is not algebraic. More precisely, if
$P\in\R[X_1,X_2,X_3,X_4]$ vanishes on $\xi_{k,l}$, then there
exists $R\in\R[X_1,X_2,X_3,X_4]$ such that
\[
 P=(X_1^2+X_2^2+X_3^2+X_4^2-1)R.
\]
\end{maintheorem}

Previous work has established degree bounds and classification
results for minimal surfaces defined within $\Sphere^3$ as the
zero set of a homogeneous polynomial.
Perdomo and Mesa~\cite{PerdomoMesa}
obtained lower bounds on the degrees of homogeneous defining
polynomials for Lawson surfaces. They also showed that surfaces
$\xi_{k,l}$ with odd $k+l$ cannot be the entire zero set in
$\Sphere^3$ of a homogeneous polynomial, since they are not
invariant under the antipodal map~\cite[Resultado~4]{PerdomoMesa}.
For $\xi_{2,2}$, Castrill\'on and Mesa~\cite{CastrillonMesa} rule
out a homogeneous defining polynomial of degree five whose gradient
is nonzero along the surface. Up to congruence, Lawson's cubic
example is the only compact immersed minimal surface in $\Sphere^3$
defined by an irreducible homogeneous cubic polynomial~\cite{Perdomo}.
An alternative proof of this classification was given by
Wang~\cite{Wang}.

Our approach uses the integrable systems description of minimal
surfaces, in which the harmonic map equation is expressed as the
flatness of an associated family of $\SL(2,\C)$-connections
$\nabla^\lambda$ on the underlying Riemann surface, with
$\lambda\in\C^*$. The immersion is reconstructed from parallel frames at the Sym
points $\lambda=\pm1$; trivial monodromy at these parameters ensures
that it is globally well defined. For tori, this approach was
developed in~\cite{Hitchin,PinkallSterling,BobenkoTori}.
The difficulty in higher genus is that the associated family
$\nabla^\lambda$ has
nonabelian monodromy for generic $\lambda$~\cite{Heller}; abelian monodromy
for every $\lambda\in\C^*$ would force the immersion to factor through
a torus~\cite{Gerding}. For Lawson surfaces, symmetries allow us to reduce the associated
family to a four-punctured sphere
$\mathcal S_4=\CP^1\setminus\{p_1,\ldots,p_4\}$.
The Dorfmeister--Pedit--Wu (DPW) method reconstructs minimal surfaces
on simply connected domains from a meromorphic matrix-valued one-form
depending on $\lambda$, called a DPW potential~\cite{DPW}.
Subsequent work developed its application to higher-genus surfaces
by addressing the global monodromy problem.
For Lawson surfaces $\xi_{k,l}$, the associated families and DPW
potentials are studied using $(k,l)$-symmetry in~\cite{HHS,BHS}.
In the case of $\xi_{g,1}$, Fuchsian DPW potentials are constructed
in~\cite{HellerHeller}; in~\cite{CHHT}, these potentials are
obtained from convergent expansions for every $g>2$.
These descriptions yield explicit area expansions for
$\xi_{g,1}$~\cite{HellerHeller,CHHT} and lead to a proof that their
area is strictly increasing with the genus $g$~\cite{CHHT}.

By uniqueness for the Bj\"orling problem, a minimal surface is
determined near a regular analytic curve by the curve and its
prescribed real analytic unit normal field along this
curve~\cite[Theorem~3.1]{MercuriOnnis}.
For Lawson surfaces, the geodesic edges used in their construction
provide canonical curves.
By computations analogous to those in the boundary potential
construction for CMC surfaces in
$\R^3$~\cite[Theorems~3.1 and~4.1]{BranderDorfmeister},
we translate these data into an explicit local normal form for the
associated family. We use this normal form to relate algebraicity
to monodromy. Similar computations comparing potential and frame
monodromy were also performed in~\cite[Section~3.1]{DorfmeisterKilian}.

\noindent\textbf{Outline of the proof.}
Let $E$ be a geodesic edge of the fundamental quadrilateral of the
Lawson surface $\xi_{k,l}$, and let $C,N$ be functions encoding the point
and unit normal along $E$. We show that these functions admit analytic
continuation as meromorphic functions on the universal cover of the
four-punctured sphere $\mathcal S_4$.
The key observation is that algebraicity of the Lawson surface
would imply a polynomial relation between $C$ and $N$.

To make use of this polynomial relation, we derive a reflection
principle (Proposition~\ref{prop:reflection}) that links the analytic
continuation of $C,N$ to the monodromy of $\nabla^\lambda$ as follows.
Choose the basepoint $z_0$ on the image of the interior of $E$
in $\mathcal S_4$.
For $\gamma\in\pi_1(\mathcal S_4,z_0)$, let
$M_\gamma^\lambda$ be the monodromy matrix, and let $\delta$ be
the antiholomorphic involution of $\mathcal S_4$ induced by Schwarz
reflection across the chosen geodesic edge. If analytic continuation
along $\gamma$ leaves $C,N$ unchanged, then
\[
 M_\gamma^\lambda=M_{\delta\gamma}^\lambda
 \qquad (|\lambda|=1),
\]
where $\delta\gamma$ denotes the reflected loop. At a puncture
$p_1$ corresponding to an umbilic of the Lawson surface, this
identity determines the exact local branch order of $(C,N)$
(Corollary~\ref{cor:local-order}). Using this identity together with
the relation~\eqref{eq:differential-relation} between $C,N$ and the Hopf
differential, we show that a polynomial relation
between $C,N$ makes both functions algebraic over $\C(z)$, where
$z$ is an affine coordinate on $\CP^1$.
Their joint analytic continuation would therefore have
finitely many branches, and the local continuation orders give
a finite unbranched cover of the Lawson surface.
The reflection identity and the
nonabelian monodromy of the associated family on this
cover~\cite{Heller} would then force, for some $\lambda_0\in S^1$,
the monodromy matrix $M_{\gamma_1}^{\lambda_0}$ of a loop
$\gamma_1$ around $p_1$ to satisfy
\[
 (M_{\gamma_1}^{\lambda_0})^2=\pm\mathrm{Id},
\]
contradicting its prescribed eigenvalues.

\textbf{The paper is organized as follows.}
Section~\ref{sec:preliminaries} fixes the algebraicity convention
and recalls the Lawson surfaces and their integrable systems construction.
Section~\ref{sec:functions} treats the meromorphic functions $C,N$
and their algebraic relations. Section~\ref{sec:monodromy} contains
the reflection principle, the passage from algebraic dependence
to finite monodromy, and the proof of the \hyperlink{thm:main}{Main Theorem}.

\section{Preliminaries}\label{sec:preliminaries}

In this section, we recall the necessary background on Lawson
surfaces and the integrable systems approach to minimal surfaces,
and fix the notation.

\subsection{Algebraic surfaces in \texorpdfstring{$\Sphere^3$}{S3}}
\label{sec:algebraicity}

\begin{definition}\label{def:algebraic}
A surface $S\subset\Sphere^3$ is called \emph{algebraic} if there exists
a polynomial $P_S\in\R[X_1,X_2,X_3,X_4]$ which vanishes on $S$ but does
not vanish identically on $\Sphere^3$.
\end{definition}

\begin{remark}\label{rem:algebraicity}
\leavevmode
\begin{enumerate}
\renewcommand{\labelenumi}{\textup{(\arabic{enumi})}}
\item The inclusion
\[
 S\subset\{x\in\Sphere^3:P_S(x)=0\}
\]
need not be an equality.
\item Since $S\subset\Sphere^3$, the polynomial
\[
 q=X_1^2+X_2^2+X_3^2+X_4^2-1
\]
vanishes on $S$ automatically. The vanishing ideal of $\Sphere^3$
in $\R[X_1,X_2,X_3,X_4]$ is generated by $q$. Thus the condition
that $P_S$ does not vanish identically on $\Sphere^3$ is
equivalent to $P_S$ not being divisible by $q$.
\end{enumerate}
\end{remark}

Hsiang~\cite{Hsiang} and Lawson~\cite{Lawson} call a surface
algebraic if the cone over it is a real algebraic set.
The existence of a polynomial $P_S$ as in Definition~\ref{def:algebraic}
is equivalent to the existence of a nonzero homogeneous polynomial
vanishing on the cone over $S$. Indeed, set $d=\deg P_S$ and,
for $x=(X_1,X_2,X_3,X_4)$, decompose the even polynomial
$P_S(x)P_S(-x)$ into homogeneous components:
\[
 P_S(x)P_S(-x)=\sum_{j=0}^d R_{2j}(x),\qquad
 \rho(x)=X_1^2+X_2^2+X_3^2+X_4^2,
\]
where $R_{2j}$ is homogeneous of degree $2j$. Then
\[
 H(x)=\sum_{j=0}^d R_{2j}(x)\rho(x)^{d-j}
\]
is homogeneous of degree $2d$ and agrees with $P_S(x)P_S(-x)$ on
$\Sphere^3$. This product is not identically zero on the sphere,
since both factors are nonzero real analytic functions on the
connected manifold $\Sphere^3$. Thus $H$ is nonzero and vanishes on
the cone $\{tx:x\in S,\ t\geq0\}$. Conversely, a nonzero homogeneous
polynomial cannot vanish identically on $\Sphere^3$: homogeneity
would then make it vanish on every nonzero point of $\R^4$.

Our definition is therefore slightly more general than that of
Hsiang and Lawson: the cone over
$S$ need only be contained in the zero set of a nonzero homogeneous
polynomial, rather than itself being a real algebraic set.

\subsection{The Lawson surfaces \texorpdfstring{$\xi_{k,l}$}{xi(k,l)}}
\label{sec:lawson}

Let $k,l\in\N$ with $k\geq l\geq1$. Set $g=kl$ and assume $g>1$.
In $\Sphere^3\subset\C^2$, define the vertices
\[
 P_1=(1,0),\quad Q_1=(0,1),\quad
 P_2=(e^{\pi i/(k+1)},0),\quad Q_2=(0,e^{\pi i/(l+1)}).
\]
Consider the geodesic quadrilateral joining the vertices
$P_1,Q_1,P_2,Q_2$ in this order. In other words, its edges are the
shorter great-circle arcs from $P_1$ to $Q_1$, from $Q_1$ to $P_2$,
from $P_2$ to $Q_2$, and from $Q_2$ to $P_1$.
Each edge has length $\pi/2$.
A solution to the Plateau problem for this boundary curve is an
area-minimizing disk whose interior is an embedded minimal surface. By the Schwarz
reflection principle, the disk extends analytically across each
boundary edge by a rotation of angle $\pi$ about the great circle
containing that edge. Each reflected disk again has geodesic
boundary edges, so the process can be repeated. The reflections
generate a finite group, and Lawson proved that the resulting
compact surface is smooth across the edges and at the vertices
and gives a minimal embedding 
\[
 f_{k,l}\colon\Sigma_{k,l}\longrightarrow\Sphere^3
\]
of a compact Riemann surface $\Sigma_{k,l}$ of genus $g$~\cite{Lawson}.
We denote the image $f_{k,l}(\Sigma_{k,l})$, together with its
conformal parametrization $f_{k,l}$, by $\xi_{k,l}$.
We identify $\Sigma_{k,l}$ with its image under the embedding $f_{k,l}$.

The Lawson immersion $f_{k,l}$ is equivariant under
$\mathbb Z_{k+1}\times\mathbb Z_{l+1}$. The action on $\Sphere^3$
is generated by the commuting rotations
\[
 (z_1,z_2)\longmapsto(e^{2\pi i/(k+1)}z_1,z_2),\qquad
 (z_1,z_2)\longmapsto(z_1,e^{2\pi i/(l+1)}z_2).
\]
The action on $\Sigma_{k,l}$ is induced via the embedding $f_{k,l}$.
We use the quotient map
\begin{equation}\label{eq:lawson-quotient}
 \pi_{k,l}\colon\Sigma_{k,l}\longrightarrow\CP^1\cong
 \begin{cases}
 \Sigma_{k,l}/\mathbb Z_{k+1},&l=1,\\
 \Sigma_{k,l}/(\mathbb Z_{k+1}\times\mathbb Z_{l+1}),&l>1
 \end{cases}
\end{equation}
taking the cyclic quotient of~\cite{HellerHeller} when $l=1$.
Its degree is $k+1$ when $l=1$, and $(k+1)(l+1)$ otherwise.
In both cases there are four branch values
$p_1,p_2,p_3,p_4$ corresponding to the umbilics of $\xi_{k,l}$.
Set
\[
 \mathcal S_4=\CP^1\setminus\{p_1,p_2,p_3,p_4\},\qquad
 \Sigma_{k,l}^\circ=\pi_{k,l}^{-1}(\mathcal S_4).
\]
Since $k,l$ are fixed, we henceforth write $f,\Sigma,\Sigma^\circ$
and $\pi$ for $f_{k,l},\Sigma_{k,l},\Sigma_{k,l}^\circ$ and
$\pi_{k,l}$, respectively.
The branch values have real cross ratio~\cite{HellerHeller,HHS},
so they lie on a common circle in
$\CP^1$. Let $x_1\in\Sigma$ satisfy $f(x_1)=Q_1$ and set
$p_1=\pi(x_1)$. The stabilizer of $Q_1=(0,1)$ in the group
defining $\pi$ has order $k+1$, so $\pi$ has local degree $k+1$
at $x_1$. We label the remaining branch values in cyclic order
on this circle.
Let $n_j$ denote the local degree of $\pi$ at each point above
$p_j$: in suitable local coordinates, the map is
$w\mapsto w^{n_j}$. These local degrees are
\begin{equation}\label{eq:quotient-local-degrees}
 n_1=n_3=k+1,\qquad
 n_2=n_4=\begin{cases}k+1,&l=1,\\l+1,&l>1.\end{cases}
\end{equation}
Since $k\geq l\geq1$ and $g=kl>1$, all the local degrees in
\eqref{eq:quotient-local-degrees} satisfy $n_j>2$.

Let $E\subset\Sphere^3$ be the edge of the quadrilateral above
joining $Q_1=f(x_1)$ to $P_1$.
Schwarz reflection across $E$ induces an antiholomorphic
involution $\delta$ of $\mathcal S_4$.
Its extension to $\CP^1$ fixes all four branch values when $l>1$.
When $l=1$, it fixes $p_1,p_3$ and interchanges $p_2,p_4$;
this is the reflection in~\cite{HellerHeller}, with the branch
values relabelled in cyclic order. In either case it preserves the
local degrees $n_j$. Choose an affine coordinate $z=x+iy$ on
$\CP^1$, with $z=\infty$ not contained in $\pi(E)$, such that all $p_j$ are
finite and
\begin{equation}\label{eq:reflection}
 \delta(z)=\overline z.
\end{equation}
The interior $E^\circ$ projects to an interval
$I=\pi(E^\circ)\subset\R$ with endpoint $p_1$.
Fix a basepoint $z_0\in I$. We also write $z_0$ for its unique
lift in $E^\circ$.
Define
\[
 p(z)=\prod_{j=1}^4(z-p_j).
\]
The puncture set is invariant under complex conjugation, so $p$
has real coefficients. For $l>1$, all four branch values are
real in this coordinate. For $l=1$, we may normalize $z$ so that
$(p_1,p_2,p_3,p_4)=(1,i,-1,-i)$~\cite{HellerHeller}, and hence
$p(z)=z^4-1$.

The Hopf differential $\langle f_{ww},\nu\rangle\dd w^2$, where
$w$ is a local holomorphic coordinate on $\Sigma$ and $\nu$ the
unit normal, is holomorphic and invariant under
$\mathbb Z_{k+1}\times\mathbb Z_{l+1}$. It therefore descends to a
meromorphic quadratic differential $Q$ on $\CP^1$.
Near $E^\circ$, we also write $z$ for the pulled-back local
coordinate $z\circ\pi$.
If $Q$ has order $d$ at $p_j$, its holomorphic pullback has order
$n_jd+2n_j-2\geq0$, so $d\geq-1$.
Since $f$ is not totally geodesic, $Q\not\equiv0$; its divisor
has degree $-4$, forcing a simple pole at each puncture and no
zeros. Thus
\begin{equation}\label{eq:hopf}
 Q=q_0\frac{\dd z^2}{p(z)},\qquad q_0\in i\R\setminus\{0\}.
\end{equation}
Since $E$ is a geodesic and $f$ is minimal,
$\langle f_{xx},\nu\rangle=\langle f_{yy},\nu\rangle=0$
along $I$. Hence
\[
 \langle f_{zz},\nu\rangle
 =-\frac{i}{2}\langle f_{xy},\nu\rangle\in i\R
 \qquad\text{along }I.
\]
Since $p$ has real coefficients, $q_0$ is purely imaginary.
On $\Sigma$, the pullback of $Q$ has order $n_j-2$ at every
point above $p_j$. Since every $n_j>2$, the umbilics
are precisely the points above $p_1,\ldots,p_4$.

\subsection{The associated family of flat connections}
\label{sec:associated-family}
The integrable systems approach studies minimal surfaces in $\Sphere^3$ through their  associated family of flat connections.
We recall the construction  following
Hitchin~\cite{Hitchin}. Let
$f\colon\Sigma\to\Sphere^3\cong\SU(2)$ be a conformal minimal
immersion of a compact Riemann surface. We write its Maurer--Cartan form as
\[
 f^{-1}\dd f=2(\Phi-\Phi^*),
 \qquad \Phi=\tfrac12\bigl(f^{-1}\dd f\bigr)^{1,0}.
\]
Here $*$ denotes the Hermitian adjoint. On the trivial rank-two
complex bundle $\Sigma\times\C^2$, define the connection $\nabla$ by
\[
 \nabla=\dd+\tfrac12f^{-1}\dd f.
\]
Then $\nabla$ is unitary, and the conditions that $f$ is conformal and an immersion
make $\Phi$ nilpotent and nowhere vanishing. The Maurer--Cartan
equation and harmonicity of $f$
then imply that the family of $\SL(2,\C)$-connections
\begin{equation}\label{eq:global-associated-family}
 \begin{aligned}
 \nabla^\lambda
 &=\nabla+\lambda^{-1}\Phi-\lambda\Phi^*, 
 \qquad \lambda\in\C^*,
 \end{aligned}
\end{equation}
is flat for every $\lambda $. We refer to these connections as the associated family of flat connections. By construction, the connections $\nabla^\lambda$ are unitary for
$\lambda\in S^1$. At the parameters $\lambda=\pm1$
(the Sym points), we have
\[
 \nabla^{\lambda=-1}=\dd,\qquad
 \nabla^{\lambda=1}=\dd+f^{-1}\dd f.
\]
Thus both connections are trivial, and the immersion $f$ can be recovered as a gauge
transformation relating them via
$\dd+f^{-1}\dd f=\dd.f=f^{-1}\circ\dd\circ f$.
For each $\lambda\in\C^*$, a local parallel frame $F^\lambda$
for $\nabla^\lambda$ is an invertible matrix-valued function satisfying
$\nabla^\lambda F^\lambda=0$.

Choosing an initial frame $F^\lambda(z_0)$ identifies
parallel transport around loops based at $z_0$ with the monodromy representation
\[
 M^\lambda\colon\pi_1(\Sigma,z_0)\longrightarrow\SL(2,\C),
 \qquad \gamma\longmapsto M_\gamma^\lambda.
\]
For loops $\gamma,\rho$ based at $z_0$ in $\Sigma$, we use the
convention that $\gamma\circ\rho$ means first following $\rho$
and then $\gamma$.
Changing the initial frame by
$F^\lambda(z_0)\mapsto F^\lambda(z_0)C$ replaces each monodromy
matrix $M_\gamma^\lambda$ by $C^{-1}M_\gamma^\lambda C$.

\subsection{Reconstruction from holomorphic potentials}\label{sec:local-dpw}

Let $U\subset\Sigma$ be a simply connected domain
containing the lifted basepoint $z_0$, and let $\mathbb D^*_\epsilon\subset\C$ be the
punctured disk centered at $\lambda=0$ of radius $\epsilon$.
A \emph{DPW potential}~\cite{DPW} on $U$ is a family $\eta^\lambda$
of holomorphic $\mathfrak{sl}(2,\C)$-valued $(1,0)$-forms.
It depends holomorphically on $\lambda\in\mathbb D^*_\epsilon$
and has a simple pole at $\lambda=0$:
\begin{equation}\label{eq:local-potential}
 \eta^\lambda=\lambda^{-1}\eta_{-1}+\eta_0+\lambda\eta_1+\cdots.
\end{equation}
The coefficient $\eta_{-1}$ is required to be nilpotent and
nowhere vanishing on $U$.
The connection $\dd+\eta^\lambda$ is then flat for all
$\lambda\in\mathbb D^*_\epsilon$.

Let $0<r<\min(1,\epsilon)$ and choose
$0<t<\min\{r,\epsilon-r\}$. Define the open annulus
\[
 W_t=\{\lambda\in\C:r-t<|\lambda|<r+t\}\subset\mathbb D^*_\epsilon,
\]
which contains the circle $|\lambda|=r$. For $\lambda\in W_t$, let $\Psi^\lambda$
be the unique solution on $U$ of
\[
 (\dd+\eta^\lambda)\Psi^\lambda=0
\]
with initial value $\Psi^\lambda(z_0)\in\SL(2,\C)$, holomorphic in $\lambda$.
The $\SU(2)$ $r$-Iwasawa decomposition~\cite{McIntosh,BHS}
exists at every point of $U$ and gives a single-valued factorization
\begin{equation}\label{eq:iwasawa}
 \Psi^\lambda=(B^\lambda)^{-1}F^\lambda,
\end{equation}
where $(B^\lambda)^{-1}$ is the positive factor and
$F^\lambda$ is the unitary factor.
Both factors are smooth in $z$.
The positive gauge $B^\lambda$ extends holomorphically and
invertibly to $\lambda=0$.
Requiring $B^{\lambda=0}$ to be upper triangular
with positive real diagonal makes the factorization unique.
The unitary factor $F^\lambda$ is $\SU(2)$-valued for
$|\lambda|=1$. It is holomorphic in
$\lambda$ on an open annulus containing
$r\leq|\lambda|\leq r^{-1}$;
see also~\cite[Section~2, Theorem~2]{DMPW}.

The positive gauge $B^\lambda$ can be used to reconstruct the associated
family $\nabla^\lambda$ from the potential using
\[
 \dd+\eta^\lambda=(\nabla^\lambda).B^\lambda
 =(B^\lambda)^{-1}\circ\nabla^\lambda\circ B^\lambda,
 \qquad 0<|\lambda|<r.
\]
The unitary factor $F^\lambda$ is a parallel frame for $\nabla^\lambda$.
Because $F^\lambda$ is unitary, there is a constant anti-linear
quaternionic structure $j$ on $\C^2$ such that
\[
 \nabla^{1/\overline\lambda}=(\nabla^\lambda).j.
\]
The family reconstructed from~\eqref{eq:iwasawa} has the required
form~\eqref{eq:global-associated-family} by the positivity of
$B^\lambda$ and the Schwarz reflection principle applied to
the reality condition above.

Identifying $\Sphere^3$ with $\SU(2)$, we define $f$ by the
Sym--Bobenko formula
\begin{equation}\label{eq:sym-formula}
 f=(F^{\lambda=1})^{-1}F^{\lambda=-1}.
\end{equation}
In the unitary trivialization of $U\times\C^2$ determined by
$F^{\lambda=-1}$,
the connections at $\lambda=-1$ and $\lambda=1$ take the form
$\dd$ and $\dd.f$, respectively. This is consistent with the
reconstruction formula in Section~\ref{sec:associated-family}.
The leading coefficient of the reconstructed connection
form is obtained by conjugating $\eta_{-1}$ by $B^{\lambda=0}$.
Since $\eta_{-1}$ is nilpotent and nowhere vanishing, $f$ is a
conformal immersion. It is minimal because the Sym points are
$\lambda=\pm1$.

Conversely, near every point $z\in\Sigma$, the associated family
of a conformal minimal immersion is gauge equivalent to a
holomorphic potential via a positive gauge defined on a sufficiently
small $\lambda$-disc around $\lambda=0$.
This gauge depends smoothly on the surface variable and is
holomorphic and invertible in $\lambda$~\cite{DPW,Heller}.

\subsection{The monodromy problem}\label{sec:global-dpw}

\textcolor{black}{The associated family $\nabla^\lambda$ on $\Sigma$ induces flat
$\SL(2,\C)$-connections on the quotient $\mathcal S_4$
in~\eqref{eq:lawson-quotient}, also denoted
$\nabla^\lambda$~\cite{HellerHeller,HHS};
see also~\cite[Lemma~5.4]{BHS}.
This family on $\mathcal S_4$ depends holomorphically on
$\lambda\in\C^*$ and has a simple pole at $\lambda=0$.
Let $F^\lambda$ be a parallel frame for $\nabla^\lambda$ on
$\widetilde{\mathcal S_4}$,
smooth in the surface variable and holomorphic in $\lambda\in\C^*$,
with $F^\lambda(z_0)\in\SU(2)$ for $|\lambda|=1$.
For a loop $\gamma\in\pi_1(\mathcal S_4,z_0)$, write
$\gamma^*F^\lambda=F^\lambda M_\gamma^\lambda$.
The matrices $M_\gamma^\lambda$ depend holomorphically on $\lambda$.
Let $\beta$ be a loop in $\Sigma^\circ$ based at $z_0$, and let
$\gamma=\pi_*\beta$ be its projection to $\mathcal S_4$.
With corresponding initial frames on $\Sigma$, we have
\begin{equation}\label{eq:monodromy-comparison}
 M_{\pi_*\beta}^\lambda
 =\varepsilon(\beta)\widetilde M_\beta^\lambda,
\end{equation}
where $\widetilde M_\beta^\lambda$ is the monodromy on $\Sigma$
along $\beta$ and
$\varepsilon\colon\pi_1(\Sigma^\circ,z_0)\to\{\pm1\}$ is the
holonomy character of a fixed flat line bundle on $\Sigma^\circ$;
see~\cite[Section~3]{HellerHeller} and~\cite[proof of Lemma~5.4]{BHS}.}

Choose generators $\gamma_1,\ldots,\gamma_4$ of
$\pi_1(\mathcal S_4,z_0)$, with $\gamma_j$ going once positively
around $p_j$ and enclosing no other puncture.
Recall the local degrees $n_j$ from~\eqref{eq:quotient-local-degrees}.

\begin{lemma}\label{thm:monodromy-conditions}
The monodromy representation of the associated family $\nabla^\lambda$
of the Lawson surface $\xi_{k,l}$ on $\mathcal S_4$ satisfies the
following properties.
\begin{enumerate}
\renewcommand{\labelenumi}{\textup{(\arabic{enumi})}}
\item $M_\gamma^\lambda\in\SU(2)$ for every
$\gamma\in\pi_1(\mathcal S_4,z_0)$ and $\lambda\in S^1$.
\item For every $\lambda\in\C^*$ and $j=1,\ldots,4$,
\[
 \operatorname{spec}(M_{\gamma_j}^\lambda)
 =\{e^{\pi i/n_j},e^{-\pi i/n_j}\}.
\]
\item At $\lambda=\pm1$, all monodromy matrices are
simultaneously diagonalizable.
\end{enumerate}
\end{lemma}

\begin{proof}
{\color{black}
For $\lambda\in S^1$, the monodromy of $\nabla^\lambda$ on
$\mathcal S_4$ inherits unitarity from $\nabla^\lambda$ on $\Sigma$. Property~\textup{(2)} follows from
\cite[proof of Lemma~5.4; see also Proposition~5.5]{BHS}.
At the Sym points, trivial monodromy on $\Sigma$
and~\eqref{eq:monodromy-comparison} give
$M_{\pi_*\beta}^\lambda\in\{\pm\mathrm{Id}\}$.
The deck group of $\pi$ is abelian, so every commutator
lifts to a closed loop in $\Sigma^\circ$. Hence
\begin{equation}\label{eq:sym-monodromy-conjugation}
 M_\gamma^\lambda M_{\gamma_1}^\lambda(M_\gamma^\lambda)^{-1}
 =\pm M_{\gamma_1}^\lambda,\qquad\lambda=\pm1.
\end{equation}
Taking traces in~\eqref{eq:sym-monodromy-conjugation} excludes the negative sign,
since $\operatorname{tr}(M_{\gamma_1}^\lambda)=2\cos(\pi/n_1)\ne0$.
Thus all monodromy matrices preserve the eigenlines of
$M_{\gamma_1}^\lambda$ and are simultaneously diagonalizable.
}
\end{proof}

We also recall the Fuchsian potential needed below.

\begin{proposition}\label{prop:potential}
\textcolor{black}{There exists an $r_0>0$ such that the Lawson surface $\xi_{k,l}$
admits a potential}
\begin{equation}\label{eq:potential}
 \eta^\lambda=\sum_{j=1}^4 A_j(\lambda)
                 \frac{\dd z}{z-p_j},\qquad
 \sum_{j=1}^4 A_j(\lambda)=0,\qquad \textcolor{black}{0<|\lambda|<r_0},
\end{equation}
with $A_j(\lambda)\in\mathfrak{sl}(2,\C)$ and the following properties:
\begin{enumerate}
\renewcommand{\labelenumi}{\textup{(\roman{enumi})}}
\item For every $j=1,\ldots,4$, $\lambda A_j(\lambda)$ is holomorphic for
$|\lambda|<r_0$. The eigenvalues of $A_j(\lambda)$ are
$\pm\frac1{2n_j}$, where $n_j$ is
defined in~\eqref{eq:quotient-local-degrees}.
\item The leading coefficient
$\eta_{-1}=\lim_{\lambda\to0}\lambda\eta^\lambda$ is a nowhere
vanishing nilpotent section of
$\End(\mathcal O^{\oplus 2})\otimes K_{\CP^1}(p_1+\cdots+p_4)$.
\item There is a single-valued positive gauge $B^\lambda$ on
$\mathcal S_4$, smooth in the surface variable and holomorphic and
invertible for $|\lambda|<r_0$, with
\begin{equation}\label{eq:global-positive-gauge}
 \dd+\eta^\lambda=(\nabla^\lambda).B^\lambda.
\end{equation}
\end{enumerate}
\end{proposition}

\begin{proof}
For $l=1$, use~\cite[Theorem~3]{HellerHeller} for the cyclic quotient.

{\color{black}
For $l>1$, let $E^0$ be the limiting holomorphic bundle of the
logarithmic extension in~\cite[proof of Proposition~5.5]{BHS}.
For the Lawson quadrilaterals described in~\cite[Section~2.3]{BHS},
the inequality $\sum_{j=1}^4 1/(2n_j)=1/(k+1)+1/(l+1)<1$
and $Q\ne0$ imply $E^0\cong\mathcal O^{\oplus2}$
by~\cite[Theorem~5.7 and its proof]{BHS}.
The proof of that theorem, together with the local construction
in~\cite[Proposition~5.2]{BHS}, establishes~\textup{(i)--(ii)}.
The positive-gauge construction in~\cite[proof of Proposition~5.5,
equations~(5.20)--(5.21)]{BHS} proves~\textup{(iii)}.
}
\end{proof}

\begin{remark}\label{rem:potential-monodromy}
For later use, set
$\Psi^\lambda=(B^\lambda)^{-1}F^\lambda$ on
$\widetilde{\mathcal S_4}$ for $0<|\lambda|<r_0$.
For sufficiently small $r$ with $0<r<\min\{r_0,1\}$, this is the
$r$-Iwasawa factorization~\eqref{eq:iwasawa}. The reconstruction from
$\eta^\lambda$ recovers $\xi_{k,l}$ (Section~\ref{sec:local-dpw}).
Since $B^\lambda$ is single-valued, $\Psi^\lambda$ has the same
monodromy as $F^\lambda$:
\[
 \gamma^*\Psi^\lambda=\Psi^\lambda M_\gamma^\lambda.
\]
\end{remark}

\section{The functions \texorpdfstring{$C$ and $N$}{C and N}}\label{sec:functions}

In this section we introduce two functions $C$ and $N$ encoding the
position and unit surface normal along a geodesic edge, and show that
algebraicity of the Lawson surface would force a polynomial relation
between them (Proposition~\ref{prop:surface-relation}).

\subsection{Definition and analytic continuation}\label{sec:continuation}

Recall from Section~\ref{sec:lawson} the chosen geodesic edge $E$
of the fundamental quadrilateral, the interval
$I=\pi(E^\circ)\subset\R$, and the basepoint $z_0\in I$.
We use the coordinate $z$ in which $\delta(z)=\overline z$,
as in~\eqref{eq:reflection}.
Choose a smaller open interval $J\subset I$ containing $z_0$
such that, in the coordinate $z$, both the immersion $f$ and its unit normal
$\nu$ are real analytic on an open neighborhood of $J$.

{\color{black}
Using $\Sphere^3\cong\SU(2)$ and $z=x+iy$, we fix the unit
normal by
\[
 \nu=\frac{f_xf^{-1}f_y}{|f_x|^2}.
\]
Choose an orthonormal basis $(e_0,e_1,e_2,e_3)$ of $\R^4$ with
$e_0=f(z_0)$, $e_1$ tangent to $E$ at $f(z_0)$ and satisfying
$p(z_0)\langle f_x(z_0),e_1\rangle>0$, and $e_2=\nu(z_0)$.
Choose the sign of $e_3$ so that $(e_1,e_3)$ is a positively
oriented tangent basis at $f(z_0)$.
The inequality fixes the direction of $e_1$: it agrees with
$f_x(z_0)$ when $p(z_0)>0$ and has the opposite direction when
$p(z_0)<0$. No further normalization of the Hopf differential
$Q$ is required.}
Along $J$, the immersion $f$ lies in the great circle spanned by
$e_0,e_1$, and $\nu$ lies in the orthogonal great circle spanned by
$e_2,e_3$. Define $\ell,\beta\colon J\to\R$ as the local real
analytic angle functions, uniquely determined by
$\ell(z_0)=\beta(z_0)=0$ and
\begin{equation}\label{eq:circle-normal}
 f=\cos\ell\,e_0+\sin\ell\,e_1,\qquad
 \nu=\cos\beta\,e_2+\sin\beta\,e_3,
\end{equation}
so that $\ell$ is an arclength parameter along $E$, while
$\beta$ is the angle of the unit normal in the fixed plane
spanned by $e_2,e_3$. The real analytic functions $e^{i\ell}$ and
$e^{i\beta}$ on $J$ have nonvanishing holomorphic extensions to a
connected open neighborhood $\mathcal U\subset\mathcal S_4$ of $J$.

\begin{definition}\label{def:functions}
Define $C,N\colon\mathcal U\to\C^*$ as the unique holomorphic
functions satisfying
\begin{equation}\label{eq:functions}
 C|_J=e^{i\ell},\qquad N|_J=e^{i\beta}.
\end{equation}
We call $C$ and $N$ the \emph{circle coordinates} of the immersion $f$
and its unit normal $\nu$, respectively.
We use the same notation for their germs at $z_0$, identifying two
holomorphic functions if they agree on some neighborhood of $z_0$.
\end{definition}

Let $\widetilde{\mathcal S_4}\to\mathcal S_4$ be the
universal covering, and fix a lift $\widetilde z_0$ of $z_0$.
We also denote by
$\delta$ the lift of the involution~\eqref{eq:reflection}
fixing $\widetilde z_0$. For a meromorphic
function $u$ on $\widetilde{\mathcal S_4}$, set
\[
 u^\#=\overline{u\circ\delta}.
\]
For a meromorphic differential $\omega$, define
$\omega^\#=\overline{\delta^*\omega}$.
For a matrix-valued function $A(z,\lambda)$, we use $\#$ to combine
reflection in $z$ and $\lambda$ with the Hermitian adjoint:
\[
 A^\#(z,\lambda)=\overline{A(\delta z,1/\overline\lambda)}^{\,T}.
\]
Since $q_0\in i\R$ and $p\in\R[z]$,
equation~\eqref{eq:hopf} gives $Q^\#=-Q$ and hence
$(2iQ)^\#=2iQ$.
For a germ $u$ at $z_0$, write $u^\gamma$ for its analytic continuation along a
loop $\gamma$ based at $z_0$.

\begin{proposition}\label{prop:continuation}
The germs $C$ and $N$ admit analytic continuation to nonconstant
meromorphic functions on $\widetilde{\mathcal S_4}$. They satisfy
\begin{equation}\label{eq:differential-relation}
 \dd\log C\,\dd\log N
 =2iQ=2iq_0\frac{\dd z^2}{p(z)},
\end{equation}
where $Q$ is the Hopf
differential in~\eqref{eq:hopf}, and
\begin{equation}\label{eq:scalar-reflection}
 C^\#=C^{-1},\qquad N^\#=N^{-1}.
\end{equation}
For each $j$, analytic continuation along $\gamma_j^{n_j}$ returns
every branch of $(C,N)$ to its initial germs. Here $n_j$ is defined
in~\eqref{eq:quotient-local-degrees}.
\end{proposition}

\begin{proof}
Along $J$, equations~\eqref{eq:circle-normal} and~\eqref{eq:functions} give
\begin{equation}\label{eq:linear-circle-coordinates}
 C=\langle f,e_0\rangle+i\langle f,e_1\rangle,\qquad
 N=\langle\nu,e_2\rangle+i\langle\nu,e_3\rangle.
\end{equation}
We use the holomorphic potential to prove that the local analytic
continuations of $f|_J$ and $\nu|_J$, and hence of $C,N$, extend
meromorphically throughout $\widetilde{\mathcal S_4}$.
Let $\eta^\lambda$ be the Fuchsian potential from
Proposition~\ref{prop:potential} for the associated family
$\nabla^\lambda$ of the Lawson surface on $\mathcal S_4$.
On $\widetilde{\mathcal S_4}$, let $F^\lambda$ be the parallel
frame of $\nabla^\lambda$ from Section~\ref{sec:global-dpw}, and set
$\Psi^\lambda=(B^\lambda)^{-1}F^\lambda$ as in
Remark~\ref{rem:potential-monodromy}. Then $\Psi^\lambda$ is a
parallel frame for $\dd+\eta^\lambda$.
The frames $\Psi^\lambda$ and $F^\lambda$ are holomorphic and
invertible in $\lambda$ for
$0<|\lambda|<r_0$ and $\lambda\in\C^*$, respectively.
\textcolor{black}{Choose $r$ sufficiently small with}
\begin{equation}\label{eq:factorization-radius}
 \textcolor{black}{0<r<\min\{r_0,1\},}
\end{equation}
and $0<t<\min\{r,r_0-r\}$.
Consider the $r$-Iwasawa factorization
$\Psi^\lambda=(B^\lambda)^{-1}F^\lambda$, with positive factor
$(B^\lambda)^{-1}$ and unitary factor $F^\lambda$.
We suppress the superscripts $\lambda$ and set $B_0=B^{\lambda=0}$.
The positive gauge $B$ is holomorphic and invertible in $\lambda$
for $|\lambda|<r_0$. Thus $B^\#$ and $\Psi^\#$ are holomorphic
and invertible for $|\lambda|>r_0^{-1}$, and $B^\#$ extends
holomorphically and invertibly to $\lambda=\infty$.
For $z\in J\subset\R$, we have $\delta z=\overline z=z$.
Thus, for $|\lambda|=1$ and $z\in J$,
$F^\#=F^*=F^{-1}$ by unitarity of $F$.
By holomorphicity in $\lambda$, $F^\#=F^{-1}$ holds for all
$\lambda\in\C^*$ and $z\in J$.
For each $z\in J$, define the following holomorphic and invertible
functions of $\lambda$ on the indicated domains:
\begin{equation}\label{eq:factorization-on-edge}
 \begin{aligned}
 P_0&=B_0^*B,
   &&|\lambda|<r_0,\\
 P_\infty&=B_0^*(B^\#)^{-1},
   &&|\lambda|>r_0^{-1}\text{, including }\lambda=\infty,\\
 \mathcal F&=B_0^*F,
   &&\lambda\in\C^*.
 \end{aligned}
\end{equation}
Substituting~\eqref{eq:factorization-on-edge} into the Iwasawa
factorization and using $\Psi^\#=F^{-1}(B^\#)^{-1}$ gives the
factorization of~\cite[Section~2, Theorem~1]{DMPW} in the form
\begin{equation}\label{eq:two-circles}
 \begin{aligned}
 \Psi&=B^{-1}F=P_0^{-1}\mathcal F
                  &&\text{on }|\lambda|=r,\\
 (\Psi^\#)^{-1}&=B^\#F=P_\infty^{-1}\mathcal F
                  &&\text{on }|\lambda|=r^{-1}.
 \end{aligned}
\end{equation}
The normalization $P_\infty^{\lambda=\infty}=\mathrm{Id}$ follows from
$(B^\#)^{\lambda=\infty}=B_0^*$ and makes this factorization unique.
Since $\eta^\lambda$ is holomorphic on $\mathcal S_4$, its parallel
frame $\Psi$ is holomorphic in $z$ on
$\widetilde{\mathcal S_4}$.
The functions $\Psi^{-1}$ on $|\lambda|=r$ and $\Psi^\#$ on
$|\lambda|=r^{-1}$ therefore depend holomorphically on $z$, so the
normalized factors $P_0,P_\infty,\mathcal F$ continue holomorphically
in $z$ to a neighborhood of $J$.
These local continuations are holomorphic in $\lambda$: $P_0$ for
$|\lambda|<r$, $P_\infty$ for $|\lambda|>r^{-1}$ (including
$\lambda=\infty$), and $\mathcal F$ for $r<|\lambda|<r^{-1}$,
with smooth boundary values on the corresponding circles.
Normalizing this factorization at $z_0$, we apply
\cite[Section~2, Theorem~3]{DMPW} to obtain analytic continuation
of these factors as meromorphic functions in $z$ on
$\widetilde{\mathcal S_4}$.
Their poles lie in a common discrete exceptional set
in $z$, independent of $\lambda$. Away from this set,
$P_0,P_\infty,\mathcal F$ are holomorphic and invertible in $\lambda$
on the domains specified above.
The theorem ensures that the singularities in $z$ arising where
the factorization fails are at most poles. In particular,
$\mathcal F^{\lambda=1}$ and $\mathcal F^{\lambda=-1}$ are
meromorphic in $z$.

By Proposition~\ref{prop:potential}(ii), in a coordinate $u$
centered at a puncture we can write $\eta_{-1}=A(u)\dd u/u$,
where $A$ is holomorphic, nilpotent and nonzero. Thus $A$ has constant rank
one, and $L=\ker\eta_{-1}$ extends across the punctures as a
holomorphic line subbundle of $\mathcal O^{\oplus 2}$ on $\CP^1$.
To express the unit surface normal $\nu$ in the meromorphic frame
$\mathcal F$, we construct the corresponding meromorphic expression
for the orthogonal projection $\Pi$ onto $B_0L$ along $J$.
Let $v=(v_1,v_2)$ be a nonzero meromorphic section of $L$.
Define the meromorphic section of
$\widetilde{\mathcal S_4}\times\C^2$ by
\begin{equation}\label{eq:section-components}
 a=(a_1,a_2)=P_0^{\lambda=0}v.
\end{equation}
Define the meromorphic matrix
\[
 \mathcal P=\frac{1}{v_1^\#a_1+v_2^\#a_2}
 \begin{pmatrix}
 a_1v_1^\#&a_1v_2^\#\\
 a_2v_1^\#&a_2v_2^\#
 \end{pmatrix}.
\]
On $J$, set
\begin{equation}\label{eq:edge-components}
 k=(k_1,k_2)=B_0v.
\end{equation}
The vector $k$ spans $B_0L$ wherever $v$ is finite and nonzero.
By~\eqref{eq:factorization-on-edge}, $P_0^{\lambda=0}=B_0^*B_0$
on $J$, so $a=B_0^*k$. Together with
$v_i^\#=\overline{v_i}$ on $J$, this gives
\begin{equation}\label{eq:kernel-projection}
 \mathcal P=B_0^*\Pi(B_0^*)^{-1},\qquad
 \Pi=\frac{kk^*}{k^*k}.
\end{equation}
On $J$, the denominator
$v_1^\#a_1+v_2^\#a_2$ of $\mathcal P$ equals
$k^*k$ in~\eqref{eq:kernel-projection}, which is positive
on the subset of $J$ where $v$ is finite and nonzero.
Thus $\mathcal P$ represents the geometric projection $\Pi$ in the
meromorphic frame $\mathcal F=B_0^*F$ along $J$.
The Sym--Bobenko formula~\eqref{eq:sym-formula} and the expression
for the unit surface normal~\cite[Theorem~14.1, equations~(14.10)--(14.11)]{Bobenko}
give, along $J$,
\begin{equation}\label{eq:meromorphic-reconstruction}
 \begin{aligned}
 f&=(\mathcal F^{\lambda=1})^{-1}\mathcal F^{\lambda=-1},\\[4pt]
 \nu&=i(\mathcal F^{\lambda=1})^{-1}(2\mathcal P-\mathrm{Id})\mathcal F^{\lambda=-1}.
 \end{aligned}
\end{equation}
Here we substituted $\mathcal F=B_0^*F$ and used
\eqref{eq:kernel-projection}, with the chosen normal orientation.
Since $\mathcal F^{\lambda=\pm1}$ and $\mathcal P$ are meromorphic on
$\widetilde{\mathcal S_4}$, these formulas give analytic continuations
of $f|_J$ and $\nu|_J$, and hence of $C$ and $N$
by~\eqref{eq:linear-circle-coordinates}.

To prove~\eqref{eq:differential-relation}, write $z=x+iy$.
Differentiating the first identity in~\eqref{eq:circle-normal} along $J$ gives
\[
 f_x=\ell_x(-\sin\ell\,e_0+\cos\ell\,e_1).
\]
Conformality gives $|f_y|=|f_x|=|\ell_x|$.
Since $f_y$ is orthogonal to $f,f_x,\nu$, we obtain
\[
 f_y=\ell_x(-\sin\beta\,e_2+\cos\beta\,e_3)
 \qquad\text{on }J,
\]
with the sign fixed by our choice of $e_3$.
Differentiating $\langle f_y,\nu\rangle=0$ gives
$\langle f_{xy},\nu\rangle=-\langle f_y,\nu_x\rangle
=-\ell_x\beta_x$. Thus, using~\eqref{eq:hopf} and the computation following it,
we obtain
\[
 \begin{gathered}
 Q=-\frac{i}{2}\langle f_{xy},\nu\rangle\dd z^2
   =\frac{i}{2}\ell_x\beta_x\,\dd z^2,\\[4pt]
 \dd\log C=i\ell_x\,\dd z,\qquad
 \dd\log N=i\beta_x\,\dd z,\\[4pt]
 \dd\log C\,\dd\log N=-\ell_x\beta_x\,\dd z^2=2iQ
 \end{gathered}
\]
along $J$, and hence everywhere by analytic continuation.
Since $Q\ne0$, both functions are nonconstant. The
identities~\eqref{eq:scalar-reflection} follow from
$|C|=|N|=1$ on $J$.

Finally, we show that $C$ and $N$ have trivial monodromy along
$\gamma_j^{n_j}$.
By Lemma~\ref{thm:monodromy-conditions}\textup{(2)}, the local
monodromy eigenvalues are
\begin{equation}\label{eq:local-eigenvalues}
 e^{\pi i/n_j},\qquad e^{-\pi i/n_j}.
\end{equation}
They are distinct and their $n_j$-th powers both equal $-1$,
so
\begin{equation}\label{eq:local-central-monodromy}
 M_{\gamma_j}^{n_j}=-\mathrm{Id}.
\end{equation}

If $\delta(p_j)=p_a$, then $n_a=n_j$ and the reflected loop
$\delta\circ\gamma_j^{n_j}$ is freely homotopic to
$\gamma_a^{-n_a}$. The monodromy matrix of $\nabla^\lambda$
along the reflected loop is therefore conjugate to
$(M_{\gamma_a}^{n_a})^{-1}=-\mathrm{Id}$
by~\eqref{eq:local-central-monodromy}.
For every $\rho\in\pi_1(\mathcal S_4,z_0)$, the monodromy along
$\rho^{-1}\gamma_j^{n_j}\rho$ and its reflection is also
$-\mathrm{Id}$, since this matrix is central.
Consequently, analytic continuation along
$\rho^{-1}\gamma_j^{n_j}\rho$ multiplies both $\Psi$ and
$(\Psi^\#)^{-1}$ in~\eqref{eq:two-circles} by $-1$.
Uniqueness leaves $P_0,P_\infty$ unchanged and replaces
$\mathcal F$ by $-\mathcal F$.
The sign cancels in~\eqref{eq:meromorphic-reconstruction}, so
both $f$ and $\nu$ are unchanged. Thus continuation along
$\gamma_j^{n_j}$ fixes every branch $(C^\rho,N^\rho)$:
\[
 (C^\rho)^{\gamma_j^{n_j}}=C^\rho,\qquad
 (N^\rho)^{\gamma_j^{n_j}}=N^\rho.\qedhere
\]
\end{proof}

\subsection{A local normal form}\label{sec:normal-form}

We next determine a local normal form of the potential $\eta^\lambda$
under a positive gauge, using the immersion and its unit normal along $E$,
analogous to the boundary potential construction for the
Bj\"orling problem in $\R^3$~\cite[Theorems~3.1 and~4.1]{BranderDorfmeister}.

Let $L=\ker\eta_{-1}$. Since $\eta_{-1}$ is nilpotent and of constant rank we have
\[
 \mathcal O^{\oplus 2}/L\cong L\otimes K_{\CP^1}(p_1+\cdots+p_4).
\]
Comparing degrees, we obtain $L\cong\mathcal O(-1)$.
After a constant gauge $K\in\SL(2,\C)$, $L$ is therefore spanned
by $(1,z)$.
Since $\eta_{-1}$ is nilpotent, has simple poles at the
punctures, and is regular at infinity, we obtain
\begin{equation}\label{eq:leading-potential}
 \eta_{-1}=\frac{c_*}{p(z)}
 \begin{pmatrix}-z&1\\-z^2&z\end{pmatrix}\dd z,
 \qquad c_*\in\C\setminus\{0\}.
\end{equation}

Under the gauge $K$, the parallel frame $\Psi$ for $\dd+\eta^\lambda$
is replaced by $K^{-1}\Psi$, with initial value $K^{-1}\Psi(z_0)$.
Together with $B\mapsto BK$, this preserves $\Psi=B^{-1}F$
and the monodromy matrices.

Choose a square root of $c_*$ and normalize the
section by $v=(v_1,v_2)=\sqrt{c_*}(1,z)$. For this normalized
section, define $a,k$ by
\eqref{eq:section-components} and~\eqref{eq:edge-components}.
Define $h$ by
\begin{equation}\label{eq:scalar-h}
 h=v_1^\#a_1+v_2^\#a_2.
\end{equation}
Since $P_0^{\lambda=0}$ is meromorphic on
$\widetilde{\mathcal S_4}$, so is $h$. Let
\[
 \widehat{\mathcal S}_4\longrightarrow\widetilde{\mathcal S_4}
\]
be the covering of degree at most two branched at the zeros
and poles of $h$ of odd order, on which $\sqrt h$ is
single-valued and meromorphic.
Along $J$, $h=k^*k>0$.
We choose the holomorphic branch of $\sqrt h$ near $J$ that is
positive on $J$.
By the identity theorem, $h^\#=h$, and this branch satisfies
$(\sqrt h)^\#=\sqrt h$ near $J$.

\begin{corollary}\label{cor:edge-normal-form}
\textcolor{black}{There is a positive gauge $G$, meromorphic in $z$ on
$\widehat{\mathcal S}_4$, such that}
\[
 (G^{\lambda=0})_{11}=\frac{v_1}{\sqrt h},\qquad
 (G^{\lambda=0})_{21}=\frac{v_2}{\sqrt h},
\]
and, for $0<|\lambda|<r$,
\begin{equation}\label{eq:edge-normal-form}
 \xi^\lambda=G^{-1}\eta^\lambda G+G^{-1}\dd G
 =\begin{pmatrix}
 0&\lambda^{-1}\omega-\vartheta\\
 \vartheta-\lambda\omega&0
 \end{pmatrix},
\end{equation}
\textcolor{black}{where}
\begin{equation}\label{eq:edge-differentials}
 \omega=\frac{1}{2i}\dd\log C=\frac{h}{p(z)}\dd z,
 \qquad \vartheta=\frac{1}{2i}\dd\log N
 =\frac{q_0}{2ih}\dd z.
\end{equation}
The right-hand side of~\eqref{eq:edge-normal-form} defines
$\xi^\lambda$ for all $\lambda\in\C^*$.
\textcolor{black}{For $r<r_0$ as in~\eqref{eq:factorization-radius},
the parallel frame $\widehat\Psi=G^{-1}\Psi$ for $\dd+\xi^\lambda$
extends holomorphically in $\lambda$ to $r<|\lambda|<r^{-1}$
away from the poles of $G$, and satisfies}
\[
 \widehat\Psi^\lambda(z)\in\SU(2)
 \qquad\text{for }z\in J\text{ and }|\lambda|=1.
\]
\par
\end{corollary}

\begin{proof}
Using the components of $a$ from~\eqref{eq:section-components}
and $h$ from~\eqref{eq:scalar-h}, define
\begin{equation}\label{eq:adapted-positive-gauge}
 G=\frac{1}{\sqrt h}P_0^{-1}
 \begin{pmatrix}
 a_1&-v_2^\#\\
 a_2&v_1^\#
 \end{pmatrix}.
\end{equation}
Equation~\eqref{eq:factorization-on-edge} and analytic continuation
give $\det P_0=1$. Hence
\[
 \det G=\frac{a_1v_1^\#+a_2v_2^\#}{h\,\det P_0}=1.
\]
Since $P_0$ is meromorphic, the analytic continuation of $G$
is meromorphic on $\widehat{\mathcal S}_4$ and has the stated values for
$G^{\lambda=0}$.
Enlarge the preimage of the exceptional set in
Proposition~\ref{prop:continuation} by the zeros and poles of $h$
and the poles of $a_1,a_2,v_1^\#,v_2^\#$.
Outside this discrete set, which is independent of $\lambda$,
formula~\eqref{eq:adapted-positive-gauge} is holomorphic for
$|\lambda|<r$ and invertible because $\det G=1$.
With $k=(k_1,k_2)$ from~\eqref{eq:edge-components}, we have
$G=B^{-1}G_{\mathrm u}$ along $J$, where
\[
 G_{\mathrm u}=\frac{1}{\sqrt h}
 \begin{pmatrix}
 k_1&-\overline{k_2}\\
 k_2&\overline{k_1}
 \end{pmatrix}\in\SU(2).
\]
Thus $\widehat\Psi=G_{\mathrm u}^{-1}F$ is unitary for
$|\lambda|=1$ along $J$.
Recall from~\eqref{eq:kernel-projection} that $\Pi$ is the orthogonal
projection onto $B_0\ker\eta_{-1}$ along $J$.
Since $G_{\mathrm u}^{-1}\Pi G_{\mathrm u}=\operatorname{diag}(1,0)$,
\eqref{eq:meromorphic-reconstruction} gives
\[
 \nu=i(\widehat\Psi^{\lambda=1})^{-1}
 \begin{pmatrix}1&0\\0&-1\end{pmatrix}
 \widehat\Psi^{\lambda=-1}\qquad\text{along }J.
\]
The immersion is recovered by~\eqref{eq:sym-formula}
with $\widehat\Psi$ in place of $F$, since $G_{\mathrm u}$ is independent
of $\lambda$.
Along $J$, the identity $G=B^{-1}G_{\mathrm u}$ gives,
for $0<|\lambda|<r$,
\[
 \dd+\xi^\lambda=(\nabla^\lambda).G_{\mathrm u}.
\]
Since $G_{\mathrm u}$ is independent of
$\lambda$, this expression contains only the powers
$\lambda^{-1},1,\lambda$ and extends to all $\lambda\in \C^*$.

By~\eqref{eq:leading-potential}, its leading coefficient has only
the upper-right entry $\omega=h\,\dd z/p(z)$, which is real
along $J$. Unitarity for $|\lambda|=1$ therefore gives
\[
 \xi^\lambda|_J=\begin{pmatrix}
 i\mu&\lambda^{-1}\omega+b\\
 -\overline b-\lambda\omega&-i\mu
 \end{pmatrix},
\]
where $\mu$ is a real one-form and $b$ a complex one-form.
With $z=x+iy$, let $T_1=f_x/\ell_x$ and $T_2=f_y/\ell_x$ along
$J$.
The parallel-frame equation, with the sign fixed by
$p(z)\ell_x>0$ on $J$, gives
\[
 \dd f=2\omega T_1.
\]
Using the parallel-frame equation for $\widehat\Psi$, we obtain along $J$
\[
 \begin{aligned}
 \dd T_1&=-2\omega f-2\operatorname{Im}(b)\nu+2\mu T_2,\\
 \dd\nu&=2\operatorname{Im}(b)T_1-2\operatorname{Re}(b)T_2.
 \end{aligned}
\]
Since $E$ is a great-circle arc, $\dd T_1=-2\omega f$.
Comparing the $\nu$- and $T_2$-components gives
$\operatorname{Im}(b)=0$ and $\mu=0$.
Thus $b$ is real, and the second equation reduces to $\dd\nu=-2bT_2$.
Comparing with $\dd f=T_1\dd\ell$ and $\dd\nu=T_2\dd\beta$
from~\eqref{eq:circle-normal}, we obtain
$\dd\ell=2\omega$ and $\dd\beta=-2b$.
Set $\vartheta=-b$. These identities hold throughout by
analytic continuation from $J$.

Equation~\eqref{eq:differential-relation} gives
$\omega\vartheta=Q/(2i)$. Substituting $\omega=h\,\dd z/p(z)$
and~\eqref{eq:hopf} yields $\vartheta=q_0\dd z/(2ih)$.
This proves~\eqref{eq:edge-normal-form} with~\eqref{eq:edge-differentials}.
By~\eqref{eq:hopf}, $q_0$ is purely imaginary. Together with $p\in\R[z]$ and $h^\#=h$,
this gives
\[
 \omega^\#=\omega,\qquad \vartheta^\#=\vartheta,
 \qquad (\xi^\lambda)^\#=-\xi^\lambda.
\]

Recall the annulus $W_t=\{\lambda\in\C:r-t<|\lambda|<r+t\}$
from Section~\ref{sec:local-dpw}.
Decrease $t$ so that $r+t<r^{-1}$.
Away from the exceptional points in the surface variable,
the matching boundary values in~\eqref{eq:two-circles} glue
$P_0\Psi$ and $\mathcal F$ holomorphically across $|\lambda|=r$.
Then $\mathcal F\Psi^{-1}$ extends $P_0$ across the same circle.
Thus, locally in the surface variable and after decreasing $t$,
the identity $\Psi=P_0^{-1}\mathcal F$ holds on $W_t$.
Substituting~\eqref{eq:adapted-positive-gauge} and
$\Psi=P_0^{-1}\mathcal F$ from~\eqref{eq:two-circles} into
$\widehat\Psi=G^{-1}\Psi$ gives on $W_t$
\begin{equation}\label{eq:annular-adapted-frame}
 \widehat\Psi
 =G^{-1}P_0^{-1}\mathcal F
 =\sqrt h
 \begin{pmatrix}
 a_1&-v_2^\#\\
 a_2&v_1^\#
 \end{pmatrix}^{-1}\mathcal F.
\end{equation}
The factor multiplying $\mathcal F$ is independent of $\lambda$,
while $\mathcal F$ is holomorphic for $r<|\lambda|<r^{-1}$.
The right-hand side therefore gives the analytic continuation of
$\widehat\Psi$ to this annulus, holomorphic in $\lambda$ and defined
at the Sym points $\lambda=\pm1$.
\end{proof}

\begin{remark}\label{rem:h-circle-coordinate}
The identity~\eqref{eq:edge-differentials} also gives the fixed
meromorphic relation
\begin{equation}\label{eq:h-circle-coordinate}
 h=\frac{p(z)}{2i}\frac{C'(z)}{C(z)}.
\end{equation}
\end{remark}

Recall from Section~\ref{sec:lawson} that the point $x_1\in\Sigma$
satisfies $f(x_1)=Q_1$ and $\pi(x_1)=p_1$.
It is an umbilic of order $n_1-2$.

\begin{lemma}\label{lem:vertex}
There is a conformal coordinate $w$ centered at $x_1$, real along
$E$, in which the initial branches of $C$ and $N$ extend
holomorphically, with $C(0)N(0)\ne0$, and
\begin{equation}\label{eq:vertex}
 z-p_1=\kappa_0 w^{n_1}+O(w^{n_1+1}),\qquad
 C-C(0)=\kappa_1w+O(w^2),
\end{equation}
for some $\kappa_0,\kappa_1\in\C\setminus\{0\}$. Moreover,
\[
 \ord_{w=0}(\dd\log N)=n_1-2.
\]
\end{lemma}

\begin{proof}
The lift to $\Sigma$ of the reflection $\delta$
in~\eqref{eq:reflection} fixing $E$ also fixes $x_1$.
Choose a conformal coordinate $w$ centered at $x_1$ in which $\delta$ acts by complex conjugation. Then $w$ is real along $E$.
Since $\pi$ has
local degree $n_1$ at $x_1$, the first expansion
in~\eqref{eq:vertex} follows.

The immersion $f$ and its unit normal $\nu$ are real analytic at
$x_1$. Along $E$, $f$ takes values in the great circle spanned by
$e_0,e_1$, while $\nu$ takes values in the great circle spanned by
$e_2,e_3$, as in~\eqref{eq:circle-normal}. By real analyticity, their restrictions
to the real $w$-axis remain in these circles through $w=0$.
Thus the angle functions $\ell$ and $\beta$
extend real analytically through $w=0$. Their holomorphic
extensions give $C=e^{i\ell}$ and $N=e^{i\beta}$ with
$C(0)N(0)\ne0$. \textcolor{black}{Since $f$ is an immersion, differentiating
\eqref{eq:circle-normal} along the real $w$-axis gives
$C'(0)=iC(0)\ell'(0)\ne0$, proving the second expansion
in~\eqref{eq:vertex}.}

The differential $Q$ has a simple pole at $p_1$, so its pullback
under $\pi$ has order $2(n_1-1)-n_1=n_1-2$ at $x_1$.
Since $C(0)$ and $C'(0)$ are nonzero,
$\dd\log C=(C'/C)\dd w$ has order zero at $w=0$.
Using~\eqref{eq:differential-relation}, we obtain
\[
 \begin{aligned}
 \ord_{w=0}(\dd\log N)
 &=\ord_{w=0}(\pi^*Q)-\ord_{w=0}(\dd\log C)\\
 &=n_1-2.
 \end{aligned}\qedhere
\]
\end{proof}

\subsection{Algebraicity of the surface and an algebraic relation}
\label{sec:algebraic-relation}

We next show that algebraicity of the surface gives a polynomial
relation between $C$ and $N$. The argument uses the fact that the
surface is locally a graph after stereographic projection.

\begin{proposition}\label{prop:surface-relation}
Let $S\subset\Sphere^3$ be an embedded real analytic surface containing
an open arc of a great circle. Let $C$ and $N$ be the corresponding
circle coordinates, as in Definition~\ref{def:functions}. If $S$ is algebraic, then $N$ is
algebraic over $\C(C)$. In other words, there is
a nonzero polynomial $A\in\C[T_1,T_2]$ such that
\[
 A(C,N)=0.
\]
\end{proposition}

\begin{proof}
Let $S$ be an algebraic surface containing an open arc of a great
circle, and let $P_S\in\R[X_1,X_2,X_3,X_4]$ vanish on $S$ but not
identically on $\Sphere^3$. Stereographic projection from a point of the great circle outside
the arc maps the circle to a straight line. After a Euclidean
change of coordinates and restriction to a smaller arc, the image
of $S$ is a real analytic graph
\[
 \textcolor{black}{z}=\chi(x,y),\qquad \chi(x,0)=0,
\]
with the arc on the $x$-axis. Let $\varphi$ denote the inverse
stereographic projection in these coordinates. The function
$P_S\circ\varphi$ is rational. Multiplying $P_S\circ\varphi$ by its polynomial denominator
gives a nonzero polynomial $R\in\R[x,y,\textcolor{black}{z}]$ satisfying
$R(x,y,\chi(x,y))=0$.

Write $R=\sum_{j=m}^d R_j$, where $R_j$ is homogeneous of degree
$j$ in $(y,\textcolor{black}{z})$ and $R_m\ne0$. Since $R(x,0,0)=0$, we have
$m\ge1$. Set $a(x)=\chi_y(x,0)$. Substituting
$\chi(x,y)=a(x)y+O(y^2)$ and comparing the coefficient of $y^m$ gives
\[
 R_m(x,1,a(x))=0.
\]
This is a nontrivial polynomial equation for $a(x)$: since $R_m$
is nonzero and homogeneous in $(y,\textcolor{black}{z})$, the polynomial $R_m(x,1,T)$
is nonzero. Thus $a$ is algebraic over $\C(x)$.
The Euclidean unit normal along the arc is
\[
 \nu_{\mathrm E}(x)=\pm\frac{(0,-a(x),1)}{\sqrt{1+a(x)^2}},
\]
so its components are algebraic over $\C(x)$ as well.

Write $\varphi^*g_{\Sphere^3}=\rho^2g_{\R^3}$.
Both $\varphi$ and $\rho>0$ are rational, and the spherical unit
normal is
\[
 \nu=\rho^{-1}\dd\varphi(\nu_{\mathrm E}).
\]
Thus the components of $\nu$ along the arc, and in particular
$N=\langle\nu,e_2\rangle+i\langle\nu,e_3\rangle$, are algebraic
over $\C(x)$. On the great circle, stereographic projection
expresses $x$ as a nonconstant rational function of $C$, since
\[
 \cos\ell=\frac{C+C^{-1}}{2},\qquad
 \sin\ell=\frac{C-C^{-1}}{2i}.
\]
Substituting this expression for $x$ into a nonzero polynomial equation
for $N$ over $\C(x)$ gives an equation whose coefficients are rational
functions of $C$. Multiplying this equation by the product of the
denominators of these coefficients gives
$A(C,N)=0$ with $A\in\C[T_1,T_2]$. The polynomial $A$ is nonzero
because $x$ is a nonconstant rational function of $C$.
Since $C$ and $N$ are holomorphic and $A(C,N)$ vanishes on an open
subarc, the same identity holds for their holomorphic germs.
\end{proof}

\section{Monodromy}\label{sec:monodromy}

In this section we complete the proof of the Main Theorem by studying
the analytic continuation of the circle coordinates $C,N$. We first
establish a reflection principle relating this analytic continuation
to the monodromy of the associated family. We then show that a
nontrivial polynomial relation between $C$ and $N$ would force the
pair to have only finitely many branches.

For $\gamma\in\pi_1(\mathcal S_4,z_0)$, write $C^\gamma,N^\gamma$
for the germs obtained by analytic continuation along $\gamma$.
With the composition convention of Section~\ref{sec:associated-family},
we write $\gamma\rho=\gamma\circ\rho$, so that we first go along $\rho$ and then $\gamma$. Thus continuation of a germ $u$ satisfies
\[
 u^{\gamma\rho}=(u^\rho)^\gamma.
\]
Recall from Remark~\ref{rem:potential-monodromy} that
\begin{equation}\label{eq:monodromy}
 \gamma^*\Psi^\lambda=\Psi^\lambda M_\gamma^\lambda
 \qquad(0<|\lambda|<r_0).
\end{equation}
With the composition convention above, the monodromy representation satisfies
\[
 M_{\gamma\rho}^\lambda=M_\gamma^\lambda M_\rho^\lambda,\qquad
 M_{\gamma^{-1}}^\lambda=(M_\gamma^\lambda)^{-1}.
\]
We omit the superscript $\lambda$ unless a particular parameter value
is specified.

Recall the antiholomorphic involution $\delta(z)=\overline z$ of
$\mathcal S_4$ induced by Schwarz reflection across $E$;
see~\eqref{eq:reflection}.
Since $\delta(\mathcal S_4)=\mathcal S_4$ and $\delta(z_0)=z_0$, reflection sends a
loop $\gamma$ based at $z_0$ to the loop $\delta\circ\gamma$, also based at $z_0$.
Thus the homotopy class $\delta\gamma:=[\delta\circ\gamma]$ is well-defined.
The reflection by $\delta$ also preserves the order in which we follow the two loops.
Hence
\[
 \delta(\gamma\rho)=(\delta\gamma)(\delta\rho),
 \qquad \delta(\gamma^{-1})=(\delta\gamma)^{-1}.
\]
Since $\delta^2=\mathrm{Id}$, this defines an involutive
automorphism of $\pi_1(\mathcal S_4,z_0)$, again denoted by $\delta$.

By Proposition~\ref{prop:continuation}, we can regard $C$ and $N$ as
meromorphic functions on the universal cover $\widetilde{\mathcal S_4}$.
Consider two lifts $\widetilde z_0$ and $\widetilde z_0'$ of
$z_0\in\mathcal S_4$ to the universal cover $\widetilde{\mathcal S_4}$,
with $\widetilde z_0$ the lift fixed in Section~\ref{sec:continuation}.
Composing $C$ and $N$ with the local inverse of the covering map,
chosen to send $z_0$ to $\widetilde z_0'$, gives a branch of
$(C,N)$ at $z_0$. If a loop $\rho$ based at $z_0$ lifts to a path
from $\widetilde z_0$ to $\widetilde z_0'$, this branch is
$(C^\rho,N^\rho)$.

\begin{definition}\label{def:finite-monodromy}
We say the pair $(C,N)$ has \emph{finite monodromy} if the set of germs
\[
 \bigl\{(C^\gamma,N^\gamma):\gamma\in\pi_1(\mathcal S_4,z_0)\bigr\}
\]
at $z_0$ is finite.
\end{definition}

\subsection{The monodromy reflection principle}\label{sec:reflection}
Recall the generating loops $\gamma_1,\ldots,\gamma_4$
\textcolor{black}{of $\pi_1(\mathcal S_4,z_0)$} from
Section~\ref{sec:global-dpw}. Since $\delta$ fixes $I$ pointwise
and reverses the orientation of a small circle centered at $p_1$,
we can choose $\gamma_1$ so that
\begin{equation}\label{eq:reflected-generator}
 \delta\gamma_1=\gamma_1^{-1}.
\end{equation}

A non-totally umbilic CMC immersion in $\R^3$ with an umbilic
has trivial dressing isotropy~\cite[Theorem~1.5]{DorfmeisterHaak}. The following related local lemma will be applied at the
umbilic $x_1$ with $f(x_1)=Q_1$ and $\pi(x_1)=p_1$.

\begin{lemma}\label{lem:positive-rigidity}
Let $a=a(z)$ and $b=b(z)$ be holomorphic functions on a disk
centered at $z=0$, with
$a(0)\ne0$, $b(0)=0$, and $b\not\equiv0$. Put
\[
 \mathcal A^\lambda=\begin{pmatrix}
 0&\lambda^{-1}a-b\\
 b-\lambda a&0
 \end{pmatrix}.
\]
Let $r>0$ and $0<t<r$. Suppose that $Z(z,\lambda)\in\SL(2,\C)$
is holomorphic for $z$ in this disk and $\lambda$ in the annulus $W_t$
from Section~\ref{sec:local-dpw}, and satisfies
\[
 \partial_zZ=[Z,\mathcal A^\lambda].
\]
If, for $z$ in a nonempty real interval in the disk, $Z$ extends
holomorphically to $|\lambda|<r$ with
\[
 Z(z,0)=\begin{pmatrix}1&*\\0&1\end{pmatrix},
\]
then $Z\equiv\mathrm{Id}$.
\end{lemma}

\begin{proof}
The Laurent coefficients of $Z$ in $\lambda$ are holomorphic
functions of $z$. On the given real interval, all coefficients of
negative powers of $\lambda$ vanish; the constant coefficient has
diagonal entries $1$ and lower-left entry $0$.
The identity theorem extends these identities to every $z$ in the
disk centered at $z=0$. Thus, for every such $z$, $Z(z,\lambda)$
extends holomorphically to $|\lambda|<r$, with
$Z(z,0)=\begin{pmatrix}1&*\\0&1\end{pmatrix}$.
We work near $z=0$, where $a$ has no zeros.
If $Z\ne\mathrm{Id}$, write
\[
 Z-\mathrm{Id}=\sum_{j\geq m}Z_j\lambda^j,
 \qquad m\geq0,\qquad Z_m\not\equiv0.
\]
The condition at $\lambda=0$ gives $\operatorname{tr}Z_m=0$
if $m=0$; if $m>0$, this follows from $\det Z=1$.
Write $Z_m=\begin{pmatrix}u&\varphi\\\psi&-u\end{pmatrix}$.
Comparison of the coefficients of $\lambda^{m-1}$ gives
\[
 0=\left[Z_m,\begin{pmatrix}0&a\\0&0\end{pmatrix}\right]
 =a\begin{pmatrix}-\psi&2u\\0&\psi\end{pmatrix}.
\]
Since $a\ne0$, we have $u=\psi=0$. Hence
\[
 Z_m=\begin{pmatrix}0&\varphi\\0&0\end{pmatrix}
\]
for a holomorphic function $\varphi=\varphi(z)$. Set
\[
 c=(Z_{m+1})_{21},\qquad
 d=(Z_{m+1})_{11}-(Z_{m+1})_{22}.
\]
The diagonal and upper-right entries in degree $m$, and the
lower-left entry in degree $m+1$, give
\[
 ac=b\varphi,\qquad ad=\varphi',\qquad c'=-bd.
\]
Thus $c=(b/a)\varphi$ and $c'=-(b/a)\varphi'$, so
\[
 \left(\frac ba \varphi^2\right)'=0.
\]
Thus $(b/a)\varphi^2$ is constant in $z$. Since $a(0)\ne0$
and $b(0)=0$, its value at $z=0$ is zero, so
$(b/a)\varphi^2\equiv0$. As $b\not\equiv0$, we obtain
$\varphi\equiv0$, contradicting the choice of $m$.
\end{proof}

Similar computations were also conducted in~\cite[Section~3.1]{DorfmeisterKilian}.
We use Lemma~\ref{lem:positive-rigidity} to prove the following reflection identity.
\par

\begin{proposition}[Reflection principle]\label{prop:reflection}
If analytic continuation along a loop
$\gamma\in\pi_1(\mathcal S_4,z_0)$ leaves both $C$ and $N$
unchanged, that is,
\[
 C^\gamma=C,\qquad N^\gamma=N,
\]
then
\begin{equation}\label{eq:reflection-monodromy}
 M_\gamma=M_{\delta\gamma}
\end{equation}
as holomorphic $\SL(2,\C)$-valued functions of $\lambda\in\C^*$.
\end{proposition}

\begin{proof}
Let $G$ and $\xi^\lambda$ be the gauge and potential from
Corollary~\ref{cor:edge-normal-form}.
\textcolor{black}{We use Lemma~\ref{lem:positive-rigidity} to show that
$G$ has monodromy $\epsilon\mathrm{Id}$ along $\gamma$, for some
$\epsilon\in\{1,-1\}$, and the same monodromy along $\delta\gamma$.
The corresponding monodromies of $\dd+\xi^\lambda$ and
$\nabla^\lambda$ therefore differ by $\epsilon$ along both loops.}

Recall that $p(z)=\prod_{j=1}^4(z-p_j)$ and, by
\eqref{eq:h-circle-coordinate},
\[
 h=\frac{p(z)}{2i}\frac{C'(z)}{C(z)}.
\]
Since $p$ is single-valued, $C^\gamma=C$ gives $h^\gamma=h$.
By~\eqref{eq:edge-differentials} and~\eqref{eq:edge-normal-form},
the identities $C^\gamma=C$ and $N^\gamma=N$ also give
$\gamma^*\xi^\lambda=\xi^\lambda$.

By definition,
$u^\#(z)=\overline{u(\delta z)}$. As $z$ follows
$\delta\gamma$, the point $\delta z$ follows $\gamma$, since
$\delta^2=\mathrm{Id}$. Therefore
\[
 (C^\#)^{\delta\gamma}=(C^\gamma)^\#=C^\#,
\]
where the second equality uses $C^\gamma=C$.
Since $C^\#=C^{-1}$, this gives
$(C^{\delta\gamma})^{-1}=C^{-1}$ and hence
$C^{\delta\gamma}=C$. The same argument gives
$N^{\delta\gamma}=N$.

Write $(\sqrt h)^\gamma=\epsilon\sqrt h$, where
$\epsilon\in\{1,-1\}$. Since $(\sqrt h)^\#=\sqrt h$ for the
chosen germ at $z_0$,
\begin{equation}\label{eq:reflected-square-root}
 (\sqrt h)^{\delta\gamma}
 =\bigl((\sqrt h)^\gamma\bigr)^\#
 =\epsilon\sqrt h,
\end{equation}
with the same sign $\epsilon$.
All continuations of $G$ are taken along representatives whose lifts
avoid the poles of every matrix entry of $G$, as well as the zeros
and poles of $h$. Since $\det G=1$, $G^{-1}$ is holomorphic
wherever $G$ is.

Choose a new basepoint $z_1\in I$ in the coordinate neighborhood of
$x_1$ from Lemma~\ref{lem:vertex}, where $G$ and its continuation
along $\gamma$ are regular. Let $\alpha$ be the path in $I$ from
$z_0$ to $z_1$. Continue the germs analytically along $\alpha$ and
replace $\gamma$ by $\alpha\gamma\alpha^{-1}$, a loop based at $z_1$.
Since $\delta$ fixes $I$ pointwise, $\delta\alpha=\alpha$, and hence
\[
 \delta(\alpha\gamma\alpha^{-1})
 =\alpha(\delta\gamma)\alpha^{-1}.
\]
Parallel transport for $\nabla^\lambda$ along $\alpha$ identifies the
monodromy representations at the two basepoints. We denote
the new basepoint by $z_0$ and the loop $\alpha\gamma\alpha^{-1}$
by $\gamma$.
We also use $J$ for a small real interval about the new
basepoint, with the identities continued along $I$.

\textcolor{black}{Let $\widehat\Psi=G^{-1}\Psi$ be the parallel frame
for $\dd+\xi^\lambda$ satisfying
$\dd\widehat\Psi^\lambda=-\xi^\lambda\widehat\Psi^\lambda$
with initial condition $\widehat\Psi^\lambda(z_0)=\mathrm{Id}$.}
Since $\xi^\lambda$ in~\eqref{eq:edge-normal-form} depends
holomorphically on $\lambda\in\C^*$, so do the normalized parallel
frame $\widehat\Psi^\lambda$ and its monodromy matrices.
Along the chosen arc of $E$, parametrized by $z\in J$,
Corollary~\ref{cor:edge-normal-form} gives
$\widehat\Psi=G_{\mathrm u}^{-1}F$, where $G_{\mathrm u}\in\SU(2)$.
Thus $\widehat\Psi^\lambda$ is unitary there for $|\lambda|=1$.

Since $\gamma^*\widehat\Psi$ and $\widehat\Psi$ solve the same
equation, write
$\gamma^*\widehat\Psi=\widehat\Psi\widehat M_\gamma$, where
$\widehat M_\gamma$ denotes the monodromy matrix of $\widehat\Psi$.
Using $\widehat\Psi=G^{-1}\Psi$ and $\gamma^*\Psi=\Psi M_\gamma$,
we also have
\begin{equation}\label{eq:normalizing-gauge-continuation}
\begin{aligned}
 \gamma^*\widehat\Psi
 &=(\gamma^*G)^{-1}\gamma^*\Psi\\
 &=(\gamma^*G)^{-1}\Psi M_\gamma
  =(\gamma^*G)^{-1}G\widehat\Psi M_\gamma.
\end{aligned}
\end{equation}
Rearranging~\eqref{eq:normalizing-gauge-continuation} gives
\begin{equation}\label{eq:gauge-comparison}
 (\gamma^*G)^{-1}G
 =\widehat\Psi\,\widehat M_\gamma M_\gamma^{-1}\widehat\Psi^{-1}
 \qquad(0<|\lambda|<r).
\end{equation}
At $z_0$, the normalization $\widehat\Psi(z_0)=\mathrm{Id}$ gives
\[
 \widehat M_\gamma M_\gamma^{-1}
 =\left((\gamma^*G)^{-1}G\right)(z_0),
\]
so $\widehat M_\gamma M_\gamma^{-1}$ is holomorphic at $\lambda=0$.
The identities
$(G^{\lambda=0})_{11}=v_1/\sqrt h$ and
$(G^{\lambda=0})_{21}=v_2/\sqrt h$ give
\[
 \left((\gamma^*G)^{-1}G\right)^{\lambda=0}=\begin{pmatrix}\epsilon&*\\0&\epsilon\end{pmatrix}.
\]
The sign here is the sign for the continuation of $\sqrt{h}$ in
\eqref{eq:reflected-square-root}.

In the coordinate $w$ of Lemma~\ref{lem:vertex}, write
$a=\omega(\partial_w)$ and $b=\vartheta(\partial_w)$.
These are holomorphic at $w=0$, with
$a(0)\ne0$ and $\ord_0b=n_1-2>0$. Thus $\xi^\lambda$ and
the initial frame $\widehat\Psi$ extend holomorphically to
$w=0$ for every $\lambda\ne0$, with $\widehat\Psi$ invertible there.

The matrix
\[
 Z=\epsilon\widehat\Psi\,\widehat M_\gamma M_\gamma^{-1}\widehat\Psi^{-1}
\]
is holomorphic at $w=0$ and satisfies
\[
 \partial_wZ=[Z,\xi^\lambda(\partial_w)].
\]
On a real interval near $0$ where $G$ and $\gamma^*G$ are regular,
$Z=\epsilon(\gamma^*G)^{-1}G$ extends holomorphically to $|\lambda|<r$,
with
\[
 Z^{\lambda=0}=\begin{pmatrix}1&*\\0&1\end{pmatrix}.
\]
Lemma~\ref{lem:positive-rigidity} gives $Z\equiv\mathrm{Id}$.
Consequently,
\[
 \gamma^*G=\epsilon G,\qquad \widehat M_\gamma=\epsilon M_\gamma.
\]
The same argument for $\delta\gamma$, using
\eqref{eq:reflected-square-root}, gives
$\widehat M_{\delta\gamma}=\epsilon M_{\delta\gamma}$ with the same sign.
Since all four monodromy matrices are holomorphic on $\C^*$,
the identity theorem gives $\widehat M_\gamma=\epsilon M_\gamma$
and $\widehat M_{\delta\gamma}=\epsilon M_{\delta\gamma}$ for every
$\lambda\in\C^*$.
In particular, both $\widehat M_\gamma$ and
$\widehat M_{\delta\gamma}$ are unitary on $S^1$.

The identity $(\xi^\lambda)^\#=-\xi^\lambda$  in
Corollary~\ref{cor:edge-normal-form} then implies that
$(\widehat\Psi^\#)^{-1}$ solves the same equation as
$\widehat\Psi$. Both are normalized to $\mathrm{Id}$ at $z_0$,
so uniqueness gives
$\widehat\Psi\widehat\Psi^\#=\mathrm{Id}$ for
$r<|\lambda|<r^{-1}$. Since $\#$ reverses the order of
matrix products, we have
\[
 \gamma^*(\widehat\Psi^\#)
 =((\delta\gamma)^*\widehat\Psi)^\#
 =\widehat M_{\delta\gamma}^\#\widehat\Psi^\#.
\]
Thus the monodromy matrices along $\gamma$ and $\delta\gamma$ satisfy
\[
 \widehat M_\gamma\widehat M_{\delta\gamma}^\#=\mathrm{Id}.
\]
On $S^1$, unitarity implies $\widehat M_{\delta\gamma}^\#=\widehat M_{\delta\gamma}^{-1}$,
so $\widehat M_\gamma=\widehat M_{\delta\gamma}$ there.
Cancelling the common sign $\epsilon$ gives
\[
 M_\gamma=M_{\delta\gamma}
\]
on $S^1$ and hence on $\C^*$ by the identity theorem.

\end{proof}

\begin{corollary}\label{cor:local-order}
For every branch $(C^\rho,N^\rho)$ at $z_0$, $n_1$ is the
smallest positive integer $m$ such that the analytic continuation of
$(C^\rho,N^\rho)$ along $\gamma_1^m$
leaves both germs unchanged. In other words, 
\[
 (C^\rho)^{\gamma_1^m}=C^\rho,\qquad
 (N^\rho)^{\gamma_1^m}=N^\rho \qquad \Rightarrow n_1\mid m.
\]

\end{corollary}

\begin{proof}
For a loop $\rho$ based at $z_0$ in $\mathcal S_4$,
consider the branch $(C^\rho,N^\rho)$ of $(C,N)$.
Suppose that continuation along $\gamma_1^m$ leaves both germs unchanged.
Following first $\rho$, then $\gamma_1$ $m$ times, and finally
$\rho$ backwards therefore leaves the original germs $(C,N)$ unchanged.
With our convention for loop composition, this is the loop
$\gamma=\rho^{-1}\gamma_1^m\rho$.
By~\eqref{eq:reflected-generator}, we can write the identity
$M_\gamma=M_{\delta\gamma}$ from
Proposition~\ref{prop:reflection} as
\[
 M_\rho^{-1}M_{\gamma_1}^{m}M_\rho
 =M_{\delta\rho}^{-1}M_{\gamma_1}^{-m}M_{\delta\rho}.
\]
At $\lambda=1$, both $M_\rho^{\lambda=1}$ and $M_{\delta\rho}^{\lambda=1}$
commute with $M_{\gamma_1}^{\lambda=1}$ by
Lemma~\ref{thm:monodromy-conditions}\textup{(3)}. The displayed identity
therefore reduces to
\[
 (M_{\gamma_1}^{\lambda=1})^m=(M_{\gamma_1}^{\lambda=1})^{-m}.
\]
Since the eigenvalues of $M_{\gamma_1}^{\lambda=1}$ are $e^{\pm\pi i/n_1}$,
we obtain $e^{2\pi i m/n_1}=1$, so $n_1$ divides $m$.
Together with the upper bound from Proposition~\ref{prop:continuation},
this proves the assertion.
\end{proof}

\subsection{From algebraic dependence to finite monodromy}\label{sec:algebraic-parameter}

In this subsection, we assume that the Lawson surface is algebraic.
By Proposition~\ref{prop:surface-relation}, $C$ and $N$ are then
algebraically dependent. We prove that both are algebraic over
$\C(z)$, so their joint analytic continuation has only finitely
many branches.

\begin{proposition}\label{prop:algebraic-parameter}
Let $C$ and $N$ be the meromorphic functions associated with the Lawson
surface $\xi_{k,l}$ in Section~\ref{sec:continuation}. If $C$ and $N$
are algebraically dependent over $\C$, then each is algebraic over
$\C(z)$. Thus there exist nonzero polynomials
$A_C,A_N\in\C[z,T]$ such that
\[
 A_C(z,C)=0,\qquad A_N(z,N)=0.
\]
In particular, the analytic continuation of $(C,N)$ has only
finitely many branches.
\end{proposition}

\begin{proof}
We construct a compact Riemann surface on which $C,N$ and the
coordinate $z$ are meromorphic.
We compare a primitive of the one-form $\alpha$
defined in~\eqref{eq:algebraic-square-root} with the elliptic integral
associated with $p$.

\textbf{Construction of the algebraic curve.}
Choose an irreducible polynomial $A\in\C[T_1,T_2]$ with $A(C,N)=0$.
Let $Y$ be the normalization of the projective closure of
$A(T_1,T_2)=0$ in $\CP^2$, and let $C_Y,N_Y$ be the meromorphic
functions induced by the coordinate functions $T_1,T_2$.
On a sufficiently small disk $D_*\subset\mathcal S_4$ around
$z_0$, the initial functions $C,N$ are holomorphic. The map
$z\mapsto(C(z),N(z))$ lifts to a holomorphic map
$\iota:D_*\to Y$ satisfying
\[
 C_Y\circ\iota=C,\qquad N_Y\circ\iota=N.
\]
Define $\pi_X:X\to Y$ as the compact normalization of the cover
defined by
\begin{equation}\label{eq:algebraic-square-root}
 \alpha^2=\frac{1}{2iq_0}\,\pi_X^*(\dd\log C_Y\,\dd\log N_Y).
\end{equation}
This is a connected cover of degree at most two; if the square
root is already well-defined on $Y$, we take $X=Y$. Set
\[
 C_X=C_Y\circ\pi_X,\qquad N_X=N_Y\circ\pi_X.
\]
After shrinking $D_*$, choose a holomorphic square root
$\sqrt{p(z)}$ and a lift $\iota_X:D_*\to X$ of $\iota$.
By~\eqref{eq:hopf} and~\eqref{eq:differential-relation}, the sign
of $\alpha$ can be chosen so that
\[
 \iota_X^*\alpha=\frac{\dd z}{\sqrt{p(z)}}.
\]

\textbf{Comparison with the elliptic integral.}
Define a local coordinate on $D_*$ by
\[
 u(z)=\int_{z_0}^z\frac{\dd\zeta}{\sqrt{p(\zeta)}},
 \qquad \dd u=\iota_X^*\alpha,
\]
\textcolor{black}{and let $R$ be its local inverse. Uniformization of the elliptic
curve $y^2=p(z)$ extends $R$ to a degree-two meromorphic function
on $\C/\Lambda$ satisfying $(R')^2=p(R)$.
Let $u_1$ be the finite limit of $u(z)$ as $z\to p_1$ along $I$.
The covering involution $u\mapsto2u_1-u$ gives}
\begin{equation}\label{eq:elliptic-branch-points}
 R^{-1}(\{p_1,p_2,p_3,p_4\})=u_1+\tfrac12\Lambda,\qquad
 R^{-1}(p_1)=u_1+\Lambda
\end{equation}
on $\C$. At these points $R$ has local degree two; elsewhere
it is locally biholomorphic as a map to $\CP^1$.

{\color{black}
Let $\pi_\circ:\widetilde{X^\circ}\to X^\circ$ be the universal
covering of the complement $X^\circ\subset X$ of the zeros and
poles of $\alpha$. Lift $\iota_X$ to
$\widetilde\iota:D_*\to\widetilde{X^\circ}$ and write
$\widetilde C_X=C_X\circ\pi_\circ$ and
$\widetilde N_X=N_X\circ\pi_\circ$.
The primitive $U$ of $\pi_\circ^*\alpha$ normalized by
$U(\widetilde\iota(z_0))=0$ satisfies $U\circ\widetilde\iota=u$,
so $R(U(\widetilde\iota(z)))=z$ on $D_*$.
On the connected set
\[
 \Omega=\widetilde{X^\circ}\setminus
 U^{-1}(u_1+\tfrac12\Lambda),
\]
$R\circ U$ is locally biholomorphic into $\mathcal S_4$.
Composing $\widetilde C_X,\widetilde N_X$ with its local inverses
gives analytic continuations of $(C,N)$. Their local monodromy around
$p_j$ has order dividing $n_j$ by Proposition~\ref{prop:continuation},
and exactly $n_1$ around $p_1$ on every branch by
Corollary~\ref{cor:local-order}.
}

\textbf{Restriction of the periods.}
For a closed curve $\gamma$ in $X^\circ$, we prove that the
period of $\alpha$ satisfies
\begin{equation}\label{eq:period-restriction}
 \tau=\int_\gamma\alpha\in\tfrac12\Lambda.
\end{equation}
The corresponding deck transformation of $\widetilde{X^\circ}$
replaces $U$ by $U+\tau$ and leaves
$\widetilde C_X,\widetilde N_X$ unchanged.

The limit $U\to u_1$ as $z\to p_1$ along $I$ corresponds to
$x_1\in\Sigma$ with $f(x_1)=Q_1$.
Let $w$ be the coordinate at $x_1$ from Lemma~\ref{lem:vertex}.
The map $w\mapsto(C(w),N(w))$ lifts uniquely to $Y$; denote the
image of $w=0$ by $a_1$, and set $C_0=C_Y(a_1)=C(0)$.
Since $C'(0)\ne0$, the lift is unramified, so $w$ is also a
local coordinate on $Y$. By the same lemma,
$\frac{1}{2iq_0}\dd\log C_Y\,\dd\log N_Y$ has order $n_1-2$
at $a_1$.
At the point of $X$ above $a_1$ selected by the initial branch,
the cover is unramified and $\alpha$ has order $(n_1-2)/2$
if $n_1$ is even. If $n_1$ is odd, a local coordinate $\zeta$
satisfies $w=\zeta^2$, and $\alpha$ has order $n_1-1$.
Thus the primitive determined by the initial branch extends
holomorphically to this point, and in a local coordinate
$\zeta$ on $X$ we have
\begin{equation}\label{eq:primitive-at-vertex}
 \begin{aligned}
 U-u_1&=a\zeta^m+O(\zeta^{m+1}),&
 C_X-C_0&=b\zeta^d+O(\zeta^{d+1}),\\
 (m,d)&=
 \begin{cases}
 (n_1/2,1),& n_1\text{ even},\\
 (n_1,2),& n_1\text{ odd},
 \end{cases}
 &a,b&\in\C\setminus\{0\}.
 \end{aligned}
\end{equation}
In either case $d/m=2/n_1\notin\mathbb Z$.

Suppose $\tau\notin\tfrac12\Lambda$, and put $z_\tau=R(u_1+\tau)$.
By~\eqref{eq:elliptic-branch-points}, $z_\tau\in\mathcal S_4$
and $R$ is locally biholomorphic at $u_1+\tau$.
The preceding continuation argument and
Proposition~\ref{prop:continuation} give a meromorphic germ $c$
at $z_\tau$ satisfying
\[
 c\bigl(R(U(\zeta)+\tau)\bigr)=C_X(\zeta).
\]
Since $C_X(\zeta)\to C_0$ as $\zeta\to0$, the germ $c$ is holomorphic at
$z_\tau$ with $c(z_\tau)=C_0$. Let $q\ge1$ be the vanishing
order of $c\circ R-C_0$ at $u_1+\tau$. Thus
\[
 c\bigl(R(u_1+\tau+t)\bigr)-C_0
 =At^q+O(t^{q+1}),\qquad A\ne0.
\]
Substituting $t=U(\zeta)-u_1=a\zeta^m+O(\zeta^{m+1})$
from~\eqref{eq:primitive-at-vertex} and using
$c(R(U(\zeta)+\tau))=C_X(\zeta)$ gives
\[
 C_X(\zeta)-C_0=Aa^q\zeta^{mq}+O(\zeta^{mq+1}).
\]
Comparing this with the leading term $b\zeta^d$ in
\eqref{eq:primitive-at-vertex} gives $d=mq$, contrary to
$d/m\notin\mathbb Z$. This proves~\eqref{eq:period-restriction}.

\textbf{Exclusion of poles.}
We next show that $\alpha$ has no poles. On $X$,
\eqref{eq:algebraic-square-root} reads
\[
 \alpha^2=\frac{1}{2iq_0}\,\dd\log C_X\,\dd\log N_X.
\]
Both logarithmic differentials have at most simple poles,
so any pole of $\alpha$ is simple.
Suppose such a pole occurs at $\zeta=0$. Choose $\varepsilon>0$
so that $0<|\zeta|<\varepsilon$ lies in $X^\circ$.
Its universal covering is
$\{\operatorname{Re}v<\log\varepsilon\}$, with covering map
$\zeta=e^v$. Choose a lift of this map to $\widetilde{X^\circ}$.
The pullback $\widetilde U$ of $U$ by this lift has the form
\[
 \widetilde U(v)=\alpha_{-1}v+\mathcal H(e^v),
 \qquad \alpha_{-1}\in\C\setminus\{0\},
\]
where $\alpha_{-1}$ is the residue of $\alpha$ and
$\mathcal H$ is holomorphic on $|\zeta|<\varepsilon$.
Choose a lattice vector $\omega_\Lambda\in\Lambda$ with
$\operatorname{Re}(\omega_\Lambda/\alpha_{-1})<0$; such a vector exists
because $\Lambda$ spans $\C$ over $\R$.
For $n\in\mathbb N$, set
\[
 v_n^0=\frac{u_1+n\omega_\Lambda-\mathcal H(0)}{\alpha_{-1}}.
\]
Then $\operatorname{Re}v_n^0\to-\infty$, and
\[
 \widetilde U(v)-u_1-n\omega_\Lambda
 =\alpha_{-1}(v-v_n^0)+\mathcal H(e^v)-\mathcal H(0).
\]
On $|v-v_n^0|=1$, the first term has modulus $|\alpha_{-1}|$,
while the remainder tends uniformly to zero.
Rouch\'e's theorem therefore gives, for sufficiently large $n$,
a point $v_n$ with $|v_n-v_n^0|<1$ and
\[
 \widetilde U(v_n)=u_1+n\omega_\Lambda,\qquad
 \operatorname{Re}v_n\to-\infty,\qquad
 \zeta_n:=e^{v_n}\to0.
\]
Moreover,
\[
 \widetilde U'(v_n)
 =\alpha_{-1}+\zeta_n\mathcal H'(\zeta_n)
 \longrightarrow\alpha_{-1}\ne0.
\]
By~\eqref{eq:elliptic-branch-points}, $R$ has local degree two
at $u_1+n\omega_\Lambda$ over $p_1$. Hence $R\circ\widetilde U$ has
local degree two at $v_n$ for sufficiently large $n$.
Analytic continuation twice around $p_1$ therefore returns
each local inverse to itself. Since $C_X(e^v)$ and $N_X(e^v)$
are single-valued meromorphic functions near $v_n$, composing
them with either inverse gives a pair whose germs return
after two turns.
The chosen lift maps a sufficiently small punctured neighborhood
of $v_n$ into $\Omega$, so these pairs are analytic continuations
of $(C,N)$. Their local monodromy has order dividing two,
contradicting Corollary~\ref{cor:local-order}, since $n_1>2$.
Thus $\alpha$ is holomorphic on $X$.

\textbf{Construction of the finite cover.}
By~\eqref{eq:period-restriction}, the periods of the holomorphic
one-form $\alpha$ define a homomorphism
\[
 \pi_1(X)\longrightarrow(\tfrac12\Lambda)/\Lambda,\qquad
 [\gamma]\longmapsto\int_\gamma\alpha\pmod\Lambda.
\]
The target has four elements, so the kernel defines a connected
unbranched covering $\widehat X\to X$ of degree at most four.
On this covering, all periods of the pullback of $\alpha$
belong to $\Lambda$. Integration therefore gives a holomorphic
map $\widehat X\to\C/\Lambda$ agreeing with $U$ modulo $\Lambda$.
Since $R$ is $\Lambda$-periodic, we can form the composition
\[
 \widehat X\xrightarrow{\,U\bmod\Lambda\,}\C/\Lambda
 \xrightarrow{\,R\,}\CP^1.
\]
The resulting meromorphic function on $\widehat X$ is denoted
by $R\circ U$. It agrees with $z$ on the initial branch and
is therefore nonconstant.

Since $\widehat X$ is compact, $R\circ U$ defines a holomorphic
map of finite degree to $\CP^1$. The field of meromorphic
functions on $\widehat X$ is therefore a finite extension of
$\C(R\circ U)$. In particular, the pullbacks of $C_X,N_X$
are algebraic over $\C(R\circ U)$. Restricting to the initial branch
shows that $C,N$ are algebraic over $\C(z)$. Hence there exist
nonzero polynomials $A_C,A_N\in\C[z,T]$ such that
\[
 A_C(z,C)=0,\qquad A_N(z,N)=0.
\]
Consequently, the pair has only finitely many branches.
\end{proof}

\subsection{Infinite monodromy of \texorpdfstring{$(C,N)$}{(C,N)}}\label{sec:no-finite-monodromy}

\begin{proposition}\label{prop:infinite-monodromy}
For $k\geq l\geq1$ with $kl>1$, the pair $(C,N)$ associated with the Lawson surface
$\xi_{k,l}$ does not have finite monodromy on $\mathcal S_4$.
\end{proposition}

\begin{proof}
Suppose that $(C,N)$ has only finitely many branches. Let
$K\subset\pi_1(\mathcal S_4,z_0)$ be the kernel of the permutation action on these
branches, and put $H=K\cap\delta(K)$. Then $H$ is a normal subgroup
of finite index, invariant under $\delta$. Every $h\in H$ fixes
the initial pair $(C,N)$, so Proposition~\ref{prop:reflection} gives
\begin{equation}\label{eq:reflection-subgroup}
 M_h=M_{\delta h}\qquad(h\in H).
\end{equation}

Let $\Sigma'{}^\circ\to\Sigma^\circ$ be the connected
covering corresponding to the finite-index subgroup
$\pi_*^{-1}(H)\subset\pi_1(\Sigma^\circ,z_0)$.
By Proposition~\ref{prop:continuation}, $\gamma_j^{n_j}$ fixes every
branch, so $\gamma_j^{n_j}\in K$ for every $j=1,\ldots,4$.
Since $\delta(\gamma_j^{n_j})$ is conjugate to
$\gamma_a^{-n_a}$, where $p_a=\delta(p_j)$, normality of $K$
gives $\gamma_j^{n_j}\in H$.
Loops around the removed points of $\Sigma$ map under $\pi$
to conjugates of these elements, which belong to $H$ by normality.
Hence the covering extends to an unbranched covering
$\widehat\pi\colon\Sigma'\to\Sigma$. If its degree is $d$, then
\[
 \operatorname{genus}(\Sigma')=1+d(g-1)>1.
\]
In particular, $f\circ\widehat\pi$ is a conformal minimal immersion of
a compact Riemann surface of genus at least two.

By~\cite[Theorem~4]{Heller}, the associated family of
$f\circ\widehat\pi$ has nonabelian monodromy for generic
$\lambda\in\C^*$.
It is nonabelian for some $\lambda_0\in S^1$.
By~\eqref{eq:monodromy-comparison},
the pullback of this representation
under the natural surjection
$\pi_1(\Sigma'{}^\circ)\to\pi_1(\Sigma')$ agrees, up to
conjugation and multiplication by scalars in $\{1,-1\}$,
with the representation
$h\mapsto M_h^\lambda$ on $H\cap\pi_*\pi_1(\Sigma^\circ)$.
These signs leave commutators unchanged, so the matrices
$M_h^{\lambda=\lambda_0}$, $h\in H$, form a
nonabelian subgroup of $\SU(2)$.

Since $H$ is a normal subgroup, $\gamma h\gamma^{-1}\in H$
for every $\gamma\in\pi_1(\mathcal S_4,z_0)$ and $h\in H$. Applying
\eqref{eq:reflection-subgroup} to this element, we obtain
\[
 \begin{aligned}
 M_\gamma M_hM_\gamma^{-1}
 &=M_{\gamma h\gamma^{-1}}
 =M_{\delta(\gamma h\gamma^{-1})}\\
 &=M_{\delta\gamma}M_{\delta h}M_{\delta\gamma}^{-1}\\
 &=M_{\delta\gamma}M_hM_{\delta\gamma}^{-1}.
 \end{aligned}
\]
Here the second line uses that $\delta$ is an automorphism of
$\pi_1(\mathcal S_4,z_0)$, and the last equality uses
\eqref{eq:reflection-subgroup} for $h$.
Consequently, $M_{\delta\gamma}^{-1}M_\gamma$ commutes with
every $M_h$, $h\in H$. At $\lambda_0$, their common centralizer
in $\SL(2,\C)$ is $\{\mathrm{Id},-\mathrm{Id}\}$.
Indeed, a nonscalar matrix in the centralizer would have a
one-dimensional eigenspace invariant under all $M_h^{\lambda=\lambda_0}$.
Unitarity would make its orthogonal complement invariant as well,
forcing the group to be abelian.

Taking $\gamma=\gamma_1$ and using~\eqref{eq:reflected-generator},
we obtain
\[
 (M_{\gamma_1}^{\lambda=\lambda_0})^2\in
 \{\mathrm{Id},-\mathrm{Id}\}.
\]
But Lemma~\ref{thm:monodromy-conditions}\textup{(2)} gives the eigenvalues
$e^{\pm\pi i/n_1}$ for $M_{\gamma_1}^{\lambda=\lambda_0}$. Their squares
$e^{\pm2\pi i/n_1}$ are distinct for $n_1>2$, a contradiction.
\end{proof}

\begin{proof}[Proof of the \hyperlink{thm:main}{Main Theorem}]
If $\xi_{k,l}$ were algebraic, Proposition~\ref{prop:surface-relation}
would give an algebraic relation between $C$ and $N$.
By Proposition~\ref{prop:algebraic-parameter}, both functions
would then be algebraic over $\C(z)$ and have finite joint
monodromy, contradicting Proposition~\ref{prop:infinite-monodromy}.
Thus every real polynomial vanishing on $\xi_{k,l}$ vanishes
identically on $\Sphere^3$.
\end{proof}

The proof of the main theorem suggests a relationship between the
algebraicity of a minimal surface and the monodromy representations
of its associated family. Since these representations are generically
nonabelian for compact minimal surfaces of genus greater than one,
we are led to the following question.
\begin{question}
Does there exist a compact embedded minimal surface in $\Sphere^3$
of genus greater than one that is algebraic?
\end{question}

\subsection*{Acknowledgements}
The first author thanks Robert Bryant for bringing Lawson's algebraicity question
to our attention during the surface theory workshop hosted by
MATRIX in Creswick in 2023, and MATRIX for providing the setting
for this exchange. We also thank Sergei Galkin for subsequently
drawing our attention back to the problem.

The main ideas and mathematical input are due to the authors.
OpenAI Codex and ChatGPT Astra were used to explore and test
these ideas and proof strategies, prepare drafts, and edit the
manuscript. The authors 
take full responsibility for the correctness of the mathematical
content and for the final text.
\par


\begin{thebibliography}{99}

\bibitem{Almgren}
F.~J.~Almgren, Jr.,
\emph{Some interior regularity theorems for minimal surfaces and an
extension of Bernstein's theorem},
Ann. of Math. (2) \textbf{84} (1966), no.~2, 277--292.

\bibitem{Bobenko}
A.~I.~Bobenko,
\emph{Constant mean curvature surfaces and integrable equations},
Russian Math. Surveys \textbf{46} (1991), no.~4, 1--45.

\bibitem{BobenkoTori}
A.~I.~Bobenko,
\emph{All constant mean curvature tori in $\R^3$, $S^3$, $H^3$
in terms of theta-functions},
Math. Ann. \textbf{290} (1991), no.~2, 209--245.

\bibitem{BHS}
A.~I.~Bobenko, S.~Heller, and N.~Schmitt,
\emph{Constant mean curvature surfaces based on fundamental
quadrilaterals},
Math. Phys. Anal. Geom. \textbf{24} (2021),
no.~4, article~37, \url{https://doi.org/10.1007/s11040-021-09397-z}.

\bibitem{BranderDorfmeister}
D.~Brander and J.~F.~Dorfmeister,
\emph{The Bj\"orling problem for non-minimal constant mean curvature
surfaces},
Comm. Anal. Geom. \textbf{18} (2010), no.~1, 171--194.

\bibitem{Brendle}
S.~Brendle,
\emph{Embedded minimal tori in $S^3$ and the Lawson conjecture},
Acta Math. \textbf{211} (2013), no.~2, 177--190.

\bibitem{Bryant}
R.~Bryant,
Answer to \emph{Is every minimal hypersurface in $S^n$ algebraic?},
MathOverflow, June~9, 2021,
\url{https://mathoverflow.net/questions/354974/}.

\bibitem{CastrillonMesa}
A.~J.~Castrill\'on V\'asquez and H.~Mesa Palomino,
\emph{Algebricidad de los ejemplos de Lawson},
Matem\'aticas: Ense\~nanza Universitaria \textbf{19} (2011), no.~1,
79--88.

\bibitem{CHHT}
S.~Charlton, L.~Heller, S.~Heller, and M.~Traizet,
\emph{Minimal surfaces and alternating multiple zetas},
preprint (2024), \url{https://arxiv.org/abs/2407.07130}.

\bibitem{DorfmeisterHaak}
J.~Dorfmeister and G.~Haak,
\emph{Investigation and application of the dressing action on surfaces
of constant mean curvature},
Q. J. Math. \textbf{51} (2000), no.~1, 57--73,
\url{https://doi.org/10.1093/qmathj/50.1.57}.

\bibitem{DorfmeisterKilian}
J.~Dorfmeister and M.~Kilian,
\emph{Dressing preserving the fundamental group},
Differential Geom. Appl. \textbf{23} (2005), no.~2, 176--204,
\url{https://doi.org/10.1016/j.difgeo.2005.06.001}.


\bibitem{DMPW}
J.~Dorfmeister, I.~McIntosh, F.~Pedit, and H.~Wu,
\emph{On the meromorphic potential for a harmonic surface in a
$k$-symmetric space},
Manuscripta Math. \textbf{92} (1997), 143--152.

\bibitem{DPW}
J.~Dorfmeister, F.~Pedit, and H.~Wu,
\emph{Weierstrass type representation of harmonic maps into symmetric
spaces},
Comm. Anal. Geom. \textbf{6} (1998), no.~4, 633--668.

\bibitem{Gerding}
A.~Gerding,
\emph{Spectral methods for higher genus constant mean curvature surfaces},
Ph.D. thesis, University of Massachusetts Amherst, 2011,
\url{https://doi.org/10.7275/5683939}.

\bibitem{HellerHeller}
L.~Heller and S.~Heller,
\emph{Fuchsian DPW potentials for Lawson surfaces},
Geom. Dedicata \textbf{217} (2023), article~101.

\bibitem{HHS}
L.~Heller, S.~Heller, and N.~Schmitt,
\emph{The spectral curve theory for $(k,l)$-symmetric CMC surfaces},
J. Geom. Phys. \textbf{98} (2015), 201--213.

\bibitem{Heller}
S.~Heller,
\emph{Higher genus minimal surfaces in $S^3$ and stable bundles},
J. Reine Angew. Math. \textbf{685} (2013), 105--122.

\bibitem{Hitchin}
N.~J.~Hitchin,
\emph{Harmonic maps from a 2-torus to the 3-sphere},
J. Differential Geom. \textbf{31} (1990), no.~3, 627--710.

\bibitem{Hsiang}
W.-Y.~Hsiang,
\emph{Remarks on closed minimal submanifolds in the standard
Riemannian $m$-sphere},
J. Differential Geom. \textbf{1} (1967), no.~3--4, 257--267,
\url{https://doi.org/10.4310/jdg/1214428093}.
\par

\bibitem{Lawson}
H.~B.~Lawson, Jr.,
\emph{Complete minimal surfaces in $S^3$},
Ann. of Math. (2) \textbf{92} (1970), no.~3, 335--374.

\bibitem{McIntosh}
I.~McIntosh,
\emph{Global solutions of the elliptic 2D periodic Toda lattice},
Nonlinearity \textbf{7} (1994), no.~1, 85--108.

\bibitem{MercuriOnnis}
F.~Mercuri and I.~I.~Onnis,
\emph{On the Bj\"orling problem in a three-dimensional Lie group},
Illinois J. Math. \textbf{53} (2009), no.~2, 431--440.

\bibitem{Perdomo}
O.~M.~Perdomo,
\emph{Characterization of order 3 algebraic immersed minimal surfaces
of $S^3$},
Geom. Dedicata \textbf{129} (2007), 23--34.

\bibitem{PerdomoMesa}
O.~M.~Perdomo and H.~Mesa P.,
\emph{Hipervariedades m\'inimas algebraicas en $S^n$},
Rev. Un. Mat. Argentina \textbf{47} (2006), no.~2, 67--75.

\bibitem{PinkallSterling}
U.~Pinkall and I.~Sterling,
\emph{On the classification of constant mean curvature tori},
Ann. of Math. (2) \textbf{130} (1989), no.~2, 407--451.

\bibitem{Wang}
J.~S.~Wang,
\emph{Degree 3 algebraic minimal surfaces in the 3-sphere},
Acta Math. Sci. Ser. B (Engl. Ed.) \textbf{32} (2012), no.~6,
2065--2084.

\end{thebibliography}
\end{document}